\documentclass[12pt,reqno]{amsart}
\usepackage{amssymb}
\usepackage{geometry}
\usepackage[latin1]{inputenc}
\usepackage[italian,english]{babel}
\usepackage{amsmath, amsfonts, amsthm}
\usepackage{graphicx}
\usepackage{pgfplots}
\usepackage{paralist}
\usepackage{mathtools, mathabx,bigints}
\numberwithin{equation}{section}
\usepackage{diagbox}
\usepackage{booktabs} % VM: Ambiente per seperatori tabelle
\usepackage{comment} % VM: Ambiente per commenti estesi
\usepackage{hyperref}
\newtheorem{thm}{Theorem}[section]
\newtheorem{lem}[thm]{Lemma}

\newtheorem{prop}[thm]{Proposition}

\newtheorem{defn}[thm]{Definition}
\theoremstyle{definition}
\newtheorem{rem}[thm]{Remark}
\theoremstyle{remark}

\newcommand{\R}{\mathbb{R}}
\newcommand{\N}{\mathbb{N}}

\DeclareMathOperator{\dive}{div}

\DeclareMathOperator{\sign}{sign}

\newenvironment{sistema}%
{\left\{\begin{array}{@{}l@{}}}{\end{array}\right.}
\patchcmd{\abstract}{\scshape\abstractname}{\textbf{\abstractname}}{}{}
\makeatletter %note a di pagina senza numero 1
\def\@makefnmark{} %note a di pagina senza numero 2
\makeatother %note a di pagina senza numero 3
 
\begin{document}

%\title{Monotonicity Principle for nonlinear electrical conductivity problems}
%\title[Monotonicity Principle for Nonlinear Inverse Problems]{Monotonicity Principle for Nonlinear Inverse Problems on Perforated domains}

\title[Piecewise nonlinear materials]
{Piecewise Nonlinear Materials and\\ Monotonicity Principle}
%{Monotonicity Principle for the Imaging of Nonlinear Homogeneous Multi-Materials}
\author[Corbo Esposito, Faella, Mottola, Piscitelli, Prakash, Tamburrino]{
Antonio Corbo Esposito$^1$, Luisa Faella$^1$, Vincenzo Mottola$^1$, Gianpaolo Piscitelli$^2$, Ravi Prakash$^3$, Antonello Tamburrino$^{1,4}$}\footnote{\\$^1$Dipartimento di Ingegneria Elettrica e dell'Informazione \lq\lq M. Scarano\rq\rq, Universit\`a degli Studi di Cassino e del Lazio Meridionale, Via G. Di Biasio n. 43, 03043 Cassino (FR), Italy.\\
$^2$Dipartimento di Scienze Economiche Giuridiche Informatiche e Motorie, Universit\`a degli Studi di Napoli Parthenope, Via Guglielmo Pepe, Rione Gescal, 80035 Nola (NA), Italy.\\ 
$^3$Departamento de Matem\'atica, Facultad de Ciencias F\'isicas y Matem\'aticas, Universidad de Concepci\'on, Avenida Esteban Iturra s/n, Bairro Universitario, Casilla 160 C, Concepci\'on, Chile.\\
$^4$Department of Electrical and Computer Engineering, Michigan State University, East Lansing, MI-48824, USA.\\
Email: corbo@unicas.it, l.faella@unicas.it, vincenzo.mottola@unicas.it, gianpaolo.piscitelli@uniparthenope.it (corresponding author), rprakash@udec.cl, antonello.tamburrino@unicas.it.}

\setcounter{tocdepth}{1}

\begin{abstract}
This paper deals with the Monotonicity Principle (MP) for nonlinear materials with piecewise growth exponent. The results obtained are relevant because they enable the use of a fast imaging method based on MP, applied to a wide class of problems with two or more materials, at least one of which is nonlinear.

The treatment is very general and makes it possible to model a wide range of practical configurations such as Superconducting (SC), Perfect Electrical Conducting (PEC) or Perfect Electrical Insulating (PEI) materials.

%The three main contribution of this paper are (i) the extension of the Monotonicity Principle stated in \cite{corboesposito2021monotonicity} also in the case when $1<p<+\infty$; (ii) the statement of the Monotonicity Principle for piecewise conductivities; (iii) the case of PEC (perfect electrical conductor) and PEI (perfect electrical insulator). We stress that this result extends \cite{corboesposito2021monotonicity}.

A key role is played by the average Dirichlet-to-Neumann operator, introduced in \cite{corboesposito2021monotonicity}, where the MP for a single type of nonlinearity was treated.

Realistic numerical examples confirm the theoretical findings.
\end{abstract}
\maketitle

\section{Introduction}
This contribution falls within the framework of a nonlinear generalization of the Calderon problem \cite{alessandrini1989remark,calderon1980inverse,calderon2006inverse}. Specifically, it consists in retrieving the electrical conductivity $\sigma=\sigma(x,|\nabla u (x)|)$ appearing in the nonlinear Laplace type equation:
 \begin{equation}\label{gproblem}
\begin{cases}
\dive\Big(\sigma (x, |\nabla u(x)|)\nabla u (x)\Big) =0\ &\text{in }\Omega\vspace{0.2cm}\\
u(x) =f(x)\qquad\  &\text{on }\partial\Omega,
\end{cases}
\end{equation}
where $f$ is the boundary data and $\Omega$ is a region filled by an electrical conducting material. The nonlinear electrical conductivity $\sigma$ in $\Omega$ is retrieved tomographically from the knowledge of boundary data evaluated on $\partial \Omega$. Specifically, the electrical conductivity is retrieved from the knowledge of the Dirichlet-to-Neumann operator \begin{equation*}
\Lambda_\sigma  :f(x)\
\mapsto\ \sigma (x, |\nabla u^f (x)|) \partial_n u^f(x),
\end{equation*}
that is the operator mapping the applied boundary potential $f$ on $\partial \Omega$, into the electrical current density, entering the domain $\Omega$ through its boundary $\partial \Omega$. As is well known, this is a strongly ill-posed inverse problem that, moreover, presents nonlinear constitutive relationships.

In the case of linear materials, i.e. $\sigma = \sigma(x)$, the Dirichlet-to-Neumann (DtN) operator $\Lambda_\sigma$ is of paramount importance for retrieving $\sigma$ from boundary data. In this framework, a relevant Monotonicity Property was found by Gisser, Isaacson and Newell in \cite{gisser1990electric}:
\begin{equation}
\label{MP_DtN}
\sigma_1(x) \leq \sigma_2(x) \text{ for a.e. }x \in \Omega \quad\Longrightarrow\quad \left\langle{\Lambda}_{\sigma_{1}} f ,f \right\rangle \leq \left\langle{\Lambda}_{\sigma_2}  f ,f \right\rangle\quad \forall f.
\end{equation}
Then, Tamburrino and Rubinacci \cite{Tamburrino_2002,Tamburrino2006FastMF} exploited the negation of \eqref{MP_DtN}:
\begin{equation*}
\exists f \ : \  \left\langle{\Lambda}_{\sigma_{1}} f ,f \right\rangle \not \leq \left\langle{\Lambda}_{\sigma_2}  f ,f \right\rangle \quad\Longrightarrow\quad \sigma_1(x) \not \leq \sigma_2(x) \text{ for a.e. }x \in \Omega 
\end{equation*}
to develop a real-time imaging method, the so-called Monotonicity Principle Method (MPM), to solve inverse obstacle problems, where the aim is to reconstruct the shape of the support of anomalies in a given background, from boundary data.

The four features rendering the MPM highly suitable for solving inverse problems are: (i) the computational cost is negligible making the method suitable for online and real-time applications, (ii) the processing can be carried out in parallel, (iii) under proper assumptions, the MPM provides rigorous upper and lower bounds to the regions occupied by the unknown anomalies, even in the presence of noise and a finite number of measurements (see \cite{Tamburrino2016284}, based on \cite{harrach2015resolution,Tamburrino_2002}), and (iv) MPM gives an exact reconstruction of the outer boundary of the unknown anomalous region, in the ideal setting of complete knowledge of the boundary measurement operator and noise free data (see \cite{albicker2020monotonicity,daimon2020monotonicity,griesmaier2018monotonicity,harrach2013monotonicity}).

%\label{2-PDE}
The MP and the related imaging method MPM are general physical/mathematical features which appear in many physical problems governed by different PDEs, from stationary PDEs to evolutive PDEs and hyperbolic evolutive PDEs.
MPM was originally developed for stationary PDEs arising from static problems (such as Electrical Resistance Tomography, Electrical Capacitance Tomography, Eddy Current Tomography and Linear Elastostatics Tomography) \cite{Calvano201232,garde2022reconstruction,harrach2013monotonicity,Tamburrino_2006,Tamburrino_2002,Tamburrino2003233}, and for stationary PDEs arising as the limit of quasi-static problems, such as Eddy Current Tomography for either large or small skin depth operations \cite{Tamburrino_2006,Tamburrino2006FastMF,Tamburrino_2010}. 
Other monotonicity-based reconstruction methods can be found in \cite{garde2022simplified} for reconstructing piecewise constant layered conductivities and in \cite{arens2023monotonicity} for the Helmholtz equation in a closed cylindrical waveguide with penetrable scattering objects.

MP was discovered and applied to tomography for problems governed by parabolic evolutive PDEs, (see \cite{Su2017,Su_2017,tamburrino2021themonotonicity,Tamburrino2015159,Tamburrino20161,Tamburrino_2016_testing}), to phenomena governed by Helmotz equations arising from wave propagation problems (hyperbolic evolutive PDEs) in time-harmonic operations (see \cite{albicker2020monotonicity,albicker2023monotonicity,daimon2020monotonicity,griesmaier2018monotonicity,AT_WAVE2015}). MP for the Helmoltz equation arising from (steady-state) optical diffuse tomography was introduced in \cite{meftahi2020uniqueness}. 
%\label{4-Regulariz.}
Regularization for MPM is introduced in \cite{garde2017convergence,Rubinacci20061179}. %Vice versa, the MP can also be used as a regularizer in \cite{harrach2016enhancing}. 

%\label{5-Applications}
The MPM has been applied to many different engineering problems. The first experimental validation of the MPM for Eddy Current Tomography is shown in  \cite{Tamburrino201226}. Other results of the MPM applied to Tomography and Nondestructive Testing, can be found in \cite{soleimani2006shape,ventre2016design}.
MPM is applied to locate breast cancer via Electrical Impedance Tomography in \cite{flores2010electrical}.
MPM is applied to the tomography of two-phase materials by means of either Electrical Capacitance Tomography, Electrical Resistance Tomography, or Magneto Inductance Tomography in \cite{soleimani2007monotonicity}.  
 Finally, MPM has also been applied to the homogenization of materials \cite{7559815} and to the inspection of concrete rebars via Eddy Current Testing \cite{DeMagistris2007389,Rubinacci2007333}.

It is worth noting that regardless of the underlying governing equation and/or the specific measurement technique, the MPM imaging method is underpinned by some form of the Monotonicity Principle.

%\label{3-NonLinear}
The subject of inverse problems with nonlinear materials is a new area of research, as clearly stated in \cite{lam2020consistency}:
\lq\lq... the mathematical analysis for inverse problems governed by nonlinear Maxwell's equations is still in the early stages of development.\rq\rq .
In this framework, the fractional $p-$biharmonic operator has been treated in \cite{kar2023fractional}, whereas the $p$-Laplacian is treated in \cite{brander2018monotonicity,carstea2020recovery,guo2016inverse,salozhong2012}.
Semilinear problems have been treated in \cite{griesmaier2022inverse} for the wave propagation equation with a nonlinear refraction index, in \cite{lin2022monotonicity} for a fractional semilinear elliptic equation with power type nonlinearities, and in \cite{HARRACH2023113188} for an elliptic equation with a piecewise analytic diffusion term.

%\cite{tatar2013monotonicity}

The most comprehensive contribution in the area of inverse problems with quasilinear constitutive relationships \eqref{gproblem} is \cite{corboesposito2021monotonicity}, where the Monotonicity Principle for the (quasilinear) elliptic case is proved. Specifically, the authors introduced the new operator $\overline{\Lambda}_\sigma$, termed as the average DtN, defined as:
\begin{equation}\label{average_DtN}% \int_{0}^{1}\Lambda  \left( \alpha f\right) \text{d}\alpha 
\overline{\Lambda}_\sigma: f\in X^p_\diamond(\partial\Omega)\mapsto  \int_{0}^{1}\Lambda_\sigma  \left( \alpha f\right) \text{d}\alpha\in X^p_\diamond(\partial\Omega)',
\end{equation}
where $X^p_\diamond(\partial\Omega)$ is a proper functional space introduced in Section \ref{unified_sec}.
The average DtN is of paramount importance because it satisfies a MP (see \cite[Th. 4.3]{corboesposito2021monotonicity}), unlike the \lq\lq classical\rq\rq \ DtN which does not satisfy the MP, apart from the very special case of the $p-$Laplacian.

A quasilinear constitutive relationship is treated in \cite{corboesposito2021monotonicity}, where the electrical conductivity has a growth exponent$^*$\footnote{$^*$We say that a function $Q(E)$ has a $p-$growth if $0<c_0\leq |\sigma(E)/E^p|\leq c_1 <+\infty$, for large $E>0$ and for some proper constants $c_0$ and $c_1$.} $2\le p<+\infty$ which is constant over the domain $\Omega$. 
The original content of the present contribution is twofold and consists in the treatment of a much larger class of materials. Specifically, we consider materials made up of different phases where, in each phase, the material has either a different growth exponent, or is a Perfect Electric Conductor (PEC), or is a Perfect Electric Insulator (PEI). Also, growth exponents $1 < p < 2$ are accounted for. The latter case is relevant for including type II supeconductors \cite{rhyner1993magnetic} in the treatment.

%The original content of this paper is threefold. Specifically, we prove prove MP when (i) the growth exponent $p$ is $1< p< 2$ and/or (ii) the growth exponent $p$ is piecewise constant over the domain $\Omega$, rather than constant as in past work \cite{corboesposito2021monotonicity}.

This mathematical generalization is relevant for real-world applications, where the growth exponent is smaller than 2, and/or it is non constant over $\Omega$, and/or some regions are well approximated by either a PEC or a PEI. For instance, type II superconductors \cite{abrikosov1957magnetic,bean1962magnetization,anderson1962theory,ginzburg1950theory} can be modeled by a growth exponent $1<p<2$. From the mathematical standpoint, this means that the electrical conductivity $\sigma$ is bounded by monotonically decreasing functions (growth exponent $p-2$), rather than monotonically increasing functions, as for $2\le p<+\infty$.
Regarding the non constancy of the growth exponent, this is a case appearing in many practical configurations involving two or more different materials. For instance, one may have a nonlinear material embedded in a linear material. In all these cases the non constant growth exponent $p$ is piecewise constant. 
Similarly, PEC (infinite electrical conductivity) and PEI (vanishing electrical conductivity) are excellent approximations of real-world applications, because the electrical conductivity varies over a large range of orders of magnitude \cite{haus1989electromagnetic}.
PEC and PEI appear naturally in nonlinear problems where the boundary data is either large or small enough (see \cite{corboesposito2023thep0laplacesignature,corboesposito2023theplaplacesignature}).

The paper is organized as follows: in Section \ref{unified_sec} we present the problem; in Section \ref{convergence_sec} we give a fundamental convergence result for the solutions; in Section \ref{connection_sec} we study the connection between the Dirichlet Energy and the Dirichlet-to-Neumann operator; in Section \ref{monotonicity_sec} we prove the Monotonicity Principles for this very general setting. In Section \ref{applications_sec} we provide numerical validations of the previous results and in Section \ref{conclusions_sec}, we present our conclusions.

\section{A unified treatment for Nonlinear Inverse Problems}\label{unified_sec}
This section presents a general and unified framework for treating the steady ohmic conduction, i.e. problem \eqref{gproblem}, in the presence of materials comprising different phases. For the sake of simplicity we refer to materials comprising two different phases, however the treatment is general and can be extended to materials consisting of multiple phases.

We first introduce the basic notations followed by the formal statement of the problem and the required assumptions. Finally, we prove the existence and uniqueness of the solution.

\subsection{Notations}
Throughout this paper, $\Omega$ is the region occupied by the conducting material. We assume $\Omega\subset\R^n$, $n\geq 2$, to be an open bounded domain (i.e. an open and connected set) with Lipschitz continuous boundary.

Hereafter we assume that the first material occupies the open bounded set $A\subset\subset\Omega$, with a Lipschitz continuous boundary and made up of a finite number of connected components $A_i$, for $i=1,...,M$. The second material occupies $B:=\Omega\setminus\overline A$, which is still a domain (see Figure \ref{fig_01_AB}).

We denote by $\mathbf{\hat{n}}$ the outer unit normal defined on $\partial\Omega$, by $\langle \cdot, \cdot\rangle$ the duality product on $L^2(\partial \Omega)$ and by $dS$ the $(n-1)$-dimensional Hausdorff measure. 
Moreover, we set
\[
L^\infty_+(\Omega):=\{\theta\in L^\infty(\Omega)\ |\ \theta\geq c_0\ \text{a.e. in}\ \Omega, \ \text{for a positive constant}\ c_0\}.
\]
Furthermore,  {the Sobolev space} $W^{1,p}_0(\Omega)$ is the closure set of $C_0^1(\Omega)$ with respect to the $W^{1,p}-$norm.

The applied boundary voltage $f$ belongs to the abstract trace space $B_{p}^{1-\frac 1p,p}(\partial\Omega)$, which, for any bounded Lipschitz open set, is a Besov space (refer to \cite{JERISON1995161,leoni17}) with the following norm:
\[
||u||_{B^{1-\frac 1p,p}(\partial\Omega)}=||u||_{L^p(\partial\Omega)}+|u|_{B^{1-\frac 1p,p}(\partial\Omega)}<+\infty,
\]
where $|u|_{B^{1-\frac 1p,p}(\partial\Omega)}$ is the Slobodeckij seminorm:
\[
|u|_{B^{1-\frac 1p,p}(\partial\Omega)}=\left(\int_{\partial\Omega}\int_{\partial\Omega}\frac{|u(x)-u(y)|^p}{||x-y||^{N-1+(1-\frac 1p)p} }dS (y)d S (x) \right)^\frac 1p,
\]
see Definition 18.32, Definition 18.36 and Exercise 18.37 in \cite{leoni17}.

This guarantees the existence of a function in $W^{1,p}(\Omega)$ whose trace is $f$ \cite[Th. 18.40]{leoni17}.

For the sake of brevity, we denote this space by $X^p(\partial \Omega)$ and its elements can be identified as the functions in $W^{1,p}(\Omega)$, modulo the equivalence relation $f\in [g]_{X^p(\partial \Omega)}$ if and only if $f-g\in W^{1,p}_0(\Omega)$, see \cite[Th. 18.7]{leoni17}.

Finally, we denote by $X^p_\diamond (\partial \Omega)$ the set of elements in $X^p(\partial \Omega)$ with vanishing integral mean on $\partial\Omega$ with respect to the measure $dS$.  
 
\subsection{Statement of the Problem}
Our aim is to prove a Monotonicity Principle for two-phase materials operating in steady ohmic conduction.

In this regime, the electric field $\bf E$ is expressed in terms of the electric scalar potential as ${\bf E}(x)=-\nabla u(x)$.

For a quasilinear material, Ohm's law ${\bf J}(x)=\sigma(x,E(x)){\bf E}(x)$ is expressed as
 \begin{equation} \label{gOhm}
 {\bf J} (x)=- \sigma (x, |\nabla u(x)|)\nabla u(x)\quad\text{in }\Omega,
 \end{equation}
where $\sigma$ is the electrical conductivity, and ${\bf J}$ is the electric current density. A quasilinear material reduces to a linear material when $\sigma$ does not depend on the electric field, i.e. $\sigma = \sigma (x)$. For a PEC we have $\bf E=\bf 0$, i.e. $\nabla u = \bf 0$, whereas for a PEI we have $\bf J=\bf 0$, i.e. $\sigma (x, |\nabla u(x)|)\nabla u(x) = \bf 0$. It is worth noting that a PEC corresponds to $\sigma = +\infty$, whereas a PEI corresponds to $\sigma = 0$.

We assume that the material in $B$ has a $p-$growth ($1 < p < +\infty$) and that the material in $A$ either has a $q-$growth, or is a PEC, or is a PEI.% For a PEC we have $\nabla u=\bf 0$, for a.e. in $A$, whereas for a PEI we have $\sigma (x, |\nabla u(x)|)\nabla u(x)=\bf 0$, for a.e. in $A$.

%with different growth exponent $p,q$ in at least two different regions. Hence, throughout this paper, we assume that $A\subset\subset\Omega$ is an open bounded set with Lipschitz boundary and a finite number of connected components, such that $B:=\Omega\setminus\overline A$ is still a domain.
%Hereafter we consider $1<p\neq q<+\infty$. Region $B$ is occupied by a conducting material with a $p-$growth whereas region $A$ is occupied by the material with a $q-$growth, see Figure \ref{fig_01_AB}. 

\begin{figure}[!ht]
	\includegraphics[width=0.8\textwidth]{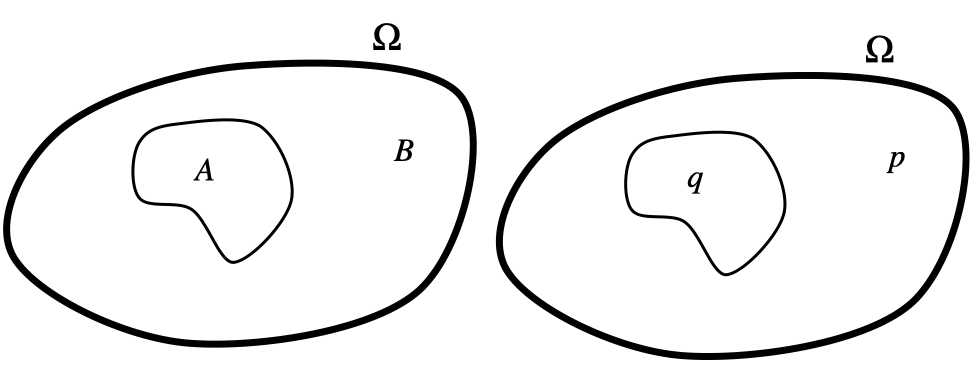}
	\caption{The outer region admits $p$ as the growth exponent, the inner region admits $q$ as the growth exponent.}
	\label{fig_01_AB}
\end{figure}

This setting is quite general from the mathematical point of view and it is very versatile from the physical/engineering viewpoint, as shown in Table \ref{tab_1_scheme}, which summarizes the cases reported in this contribution. Specifically, (i) a growth exponent smaller than 2 corresponds to an electrical conductivity monotonically decreasing with the magnitude of the electric field $E$, (ii) a growth exponent equal to 2 corresponds to an electrical conductivity independent of $E$, i.e. a linear material, and (iii) a growth exponent larger than 2 corresponds to an electrical conductivity increasing with $E$ (see Section \ref{sub_sec_assumptions}).

The setting also includes cases where the electrical conductivity is increasing or constant with the electric field in one phase and is decreasing in the other, for instance $p \ge 2$ and $1<q<2$. Other cases included in this treatment refer to materials where the order relation between the electrical conductivities in regions $A$ and $B$ changes with the value of the magnitude of the electric field $E$. Also, $p=q$ corresponds to $A=\emptyset$. In addition, for $p=q=2$ this case corresponds to \cite{gisser1990electric,Tamburrino_2002}, whereas for $p=q>2$ it corresponds to \cite{corboesposito2021monotonicity}.

%Refer also to Figure \ref{fig_05_two_conductivities} for the electrical conductivities of a copper material with an interior superconductor and Figure \ref{fig_06_anomaly}.

\begin{table}[h]
\centering
\begin{tabular}{|c|c|c|c|c|c|}
\hline \backslashbox{\qquad $B$}{$A$\qquad} & $1<q<2$ & $q=2$ & $2<q<+\infty$ & PEC & PEI\\
\hline $1<p<2$ & $\bullet$ & $\bullet$ & $\bullet$& $\bullet$ & $\bullet$\\ 
$p=2$ & $\bullet$ & \cite{gisser1990electric,Tamburrino_2002} & $\bullet$ & $\bullet$ & $\bullet$ \\
$2<p<+\infty$ & $\bullet$ & $\bullet$ & $\bullet$, (\cite{corboesposito2021monotonicity} for $p=q$) & $\bullet$ & $\bullet$\\
\hline
\end{tabular}
\caption{Scheme of all cases reported within this general framework. The symbol $\bullet$ stands for \lq\lq yes\rq\rq, i.e. the specific combination of materials in phases $A$ and $B$ can be treated within this setting. References are reported for those specific cases treated in past works.}
\label{tab_1_scheme}
\end{table}

%\begin{figure}[!ht]	\includegraphics[width=0.6\textwidth]{MPM_05_two_conductivities.png}	\caption{For small value of $E$, the electric conductivity of the copper $\sigma_B(\cdot,E)$ is above the electric conductivity of the superconductor $\sigma_A(\cdot,E)$. The order relation reverse for large value of $E$.}	\label{fig_05_two_conductivities}\end{figure}
%\begin{figure}[!ht]	\includegraphics[width=0.3\textwidth]{MPM_06_anomalyI.png}	\includegraphics[width=0.3\textwidth]{MPM_06_anomalyII.png}\includegraphics[width=0.3\textwidth]{MPM_06_anomalyIII.png}\includegraphics[width=0.3\textwidth]{MPM_06_anomalyIV.png}\includegraphics[width=0.3\textwidth]{MPM_06_anomalyV.png}\includegraphics[width=0.3\textwidth]{MPM_06_anomalyVI.png}	\caption{The reciprocal order relation between $\sigma_B(\cdot,E)$ and $\sigma_A(\cdot,E)$ can vary depending on the values of $E$.}\label{fig_06_anomaly}\end{figure}

\subsection{The assumptions} \label{sub_sec_assumptions}
The well-posedness of the forward problem in the sense of Hadamard is the minimal requirement to formulate the associated inverse problem. This is guaranteed by the following assumptions on $\sigma_B: B\times[0,+\infty)\to\R$ and $\sigma_A: A\times[0,+\infty)\to\R$.

Firstly, we recall the definition of the Carath\'eodory function.
\begin{defn}
$\sigma:\Omega\times[0,+\infty)\to\R$ is a Carath\'eodory function iff:
\begin{itemize}
\item $x\in\overline\Omega\mapsto \sigma(x,E)$ is Lebesgue-measurable for every $E\in[0,+\infty)$,
\item $E\in [0,+\infty)\mapsto \sigma(x, E)$ is continuous for almost every $x\in\Omega$.
\end{itemize}
\end{defn}

\begin{itemize}
\item[{\bf (P1)}] $\sigma_B$ and $\sigma_A$ are Carath\'eodory function. 
%\item[{\bf (A2)}]
%$E\in\ [0,+\infty)\mapsto \sigma_B(x, E) E$ is strictly increasing for a.e. $x \in B$,\\ 
%$E\in\ [0,+\infty)\mapsto \sigma_A(x, E) E$ is strictly increasing for a.e. $x \in A$.
\item[{\bf (P2)}] For fixed $1<p<+\infty$, there exists three positive constants $\underline{\sigma}\le\overline{\sigma}$ and $E_0$ 
such that: 
\[
\begin{split}
&\underline{\sigma} \left(\frac{E}{E_0}\right)^{p-2}\leq\sigma_{B}(x, E)\leq \overline{\sigma}\left[1+\left( \frac{E}{E_0} \right)^{p-2}\right]\qquad \text{if}\ p\ge 2,\\
&\underline{\sigma} \left(\frac{E}{E_0}\right)^{p-2}\leq\sigma_{B}(x, E)\leq \overline{\sigma}\left( \frac{E}{E_0} \right)^{p-2}\qquad\qquad\quad\   \text{if}\ 1<p< 2,\\
\end{split}
\]
$\text{for a.e.}\ x\in {\overline B}\ \text{and}\ \forall E\ge 0$.
\item[{\bf (P3)}] For fixed $1<p <+\infty$, there exists $\kappa>0$ such that:
\begin{equation*}
\begin{split}
(\sigma_B(x,E_2)&{\bf E}_2-\sigma_B(x,E_1){\bf E}_1)\cdot( {\bf E}_2-{\bf E}_1)\\
        &\geq
        \begin{cases}
            \kappa|{\bf E}_2-{\bf E}_1|^p\ &\text{if} \ p\geq 2\\
            \kappa(1+ |{\bf E}_2|^2+|{\bf E}_1|^2)^\frac{p-2}2|{\bf E}_2-{\bf E}_1|^2\ &\text{if}\ 1<p<2
        \end{cases}
        \\
        \end{split}
       \end{equation*}
       \ $\text{for a.e.}\ x\in B$, and for any ${\bf E}_1,{\bf E}_2\in\R^n$.
\end{itemize}

\begin{itemize}
%\item[{\bf (QX)}] $A=\emptyset$;
\item[{\bf (Q1)}] For fixed $1< q<+\infty$, there exists three positive constants $\underline{\sigma}\le\overline{\sigma}$ and $E_0$
such that:  
\[
\begin{split}
&\underline{\sigma} \left(\frac{E}{E_0}\right)^{q-2}\leq\sigma_{A}(x, E)\leq \overline{\sigma}\left[1+\left( \frac{E}{E_0} \right)^{q-2}\right] \qquad \text{if}\ q\ge 2,\\ &\underline{\sigma} \left(\frac{E}{E_0}\right)^{q-2}\leq\sigma_{A}(x, E)\leq \overline{\sigma}\left( \frac{E}{E_0} \right)^{q-2}\qquad\qquad\quad\ \text{if}\ 1<q< 2,
\end{split}
\]
$\text{for a.e.}\ x\in A\ \text{and}\ \forall E>0$.
\item[{\bf (Q2)}] For fixed $1< q<+\infty$, there exists a positive constant $\kappa$ 
such that: 
\begin{equation*}
\begin{split}(\sigma_A(x,E_2)&{\bf E}_2-\sigma_A(x,E_1){\bf E}_1)\cdot( {\bf E}_2-{\bf E}_1)      \\ 
&\geq
        \begin{cases}
            \kappa|{\bf E}_2-{\bf E}_1|^q\ &\text{if} \ q\geq 2\\
            \kappa(1+ |{\bf E}_2|^2+|{\bf E}_1|^2)^\frac{q-2}2|{\bf E}_2-{\bf E}_1|^2\ &\text{if}\ 1<q<2
        \end{cases}
        \\
\end{split}
\end{equation*}$\ \text{for a.e.}\ x\in A,$
for any ${\bf E}_1,{\bf E}_2\in\R^n$.

\item[{\bf (R)}] $\sigma_{A}(x, E)=0$,  $\text{for a.e.}\ x\in A$.
\item[{\bf (S)}] $\sigma_{A}(x, E)=+\infty$,  $\text{for a.e.}\ x\in A$.
\end{itemize}

Assumptions (QX), (R) and (S) are alternative. Throughout the paper, we will treat all the cases arising from these assumptions. Moreover, let us remark that $A$ could be $\emptyset$; in this case the assumptions (QX), (R) and (S) are not considered.

Figure \ref{fig_02_sigma} represents the geometrical interpretations of (P2). Similar arguments are valid for (Q1).

\begin{figure}[!ht]
	\centering
\includegraphics[width=0.45\textwidth]{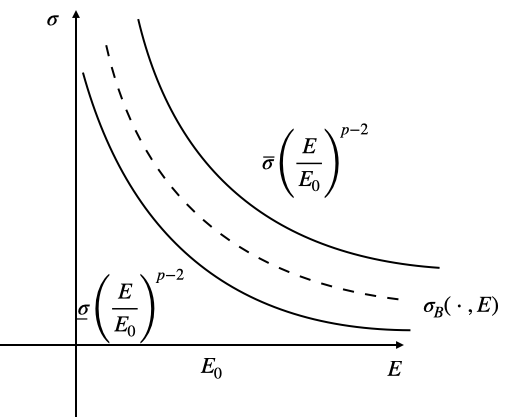}
\includegraphics[width=0.45\textwidth]{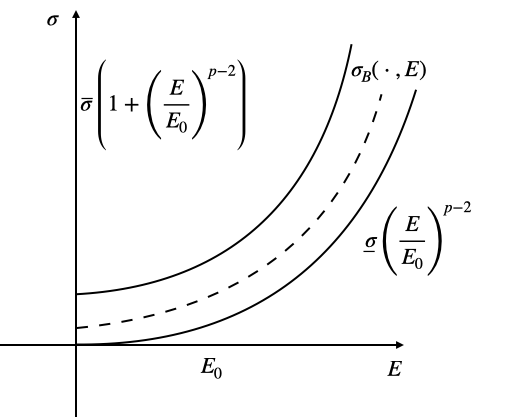}
\caption{The impact of constraints in (P2) on electrical conductivity, for $1<p<2$ (left) and $p\ge 2$ (right). Dashed lines represent the upper and lower bounds, whereas the solid line represents the actual electrical conductivity.}
\label{fig_02_sigma}
\end{figure}

It is worth nothing that assumptions (PX), (QX), (R) and (S) make it possible to include many constitutive relationship encountered in practical applications. In the following the meaning of the main assumptions is described.

\begin{rem}
Assumption (P1) requires the constitutive relationship of the material to be continuous in the electric field $E$, in both regions $A$ and $B$.
\end{rem}
\begin{rem}
Assumptions (P2) and (Q1) imply that
\begin{equation}
    \label{min_max_estimates}
\underline{\sigma} \left(\frac{E}{E_0}\right)^{r-2}\leq\sigma (x,E)\leq \overline{\sigma}\left[1+\left(\frac{E}{E_0}\right)^{s-2}\right]
\end{equation}
for a.e. $x\in \Omega$ and for all $E>0$, as $r=\min\{p,q\}$ and $s=\max\{p,q\}$.

In other terms, the behaviour of $\sigma$ as a function of the electric field $E$, is bounded by two proper order monomials.
\end{rem}

% \[
% \underline{\sigma} \min\left\{\left(\frac{E}{E_0}\right)^{p-2},\left(\frac{E}{E_0}\right)^{q-2}\right\}\leq\sigma(x, E)\leq \overline{\sigma}\left(1+\max\left\{\left(\frac{E}{E_0}\right)^{p-2},\left(\frac{E}{E_0}\right)^{q-2}\right\}\right)
% \]
% for a.e. $x\in \Omega$ and for all $E>0$. By setting $r=\min\{p,q\}$ and $s=\max\{p,q\}$, we remark that in equivalent way we have the following estimates

\begin{rem} \label{rem_ipotesi_lin}
(P3) and (Q2) imply that the constitutive relationship, expressed in term of the function $E \mapsto \sigma(x,E)E$, is monotonically increasing. Therefore, the map $E \mapsto Q_\sigma(x,E)$ is convex and is mandatory to prove the Monotonicity Principle for the average DtN.

Moreover, the first inequality in assumptions (P3) and (Q2) is simply a generalization of the standard inequality (see \cite[Sec.12, eq (I)]{lindqvist2019notes} and \cite[Lem. 2.2]{brasco2013anisotropic})
\begin{equation*}
(E_2^{k}\ {\bf E}_2-E_1^{k}\ {\bf E}_1) \cdot ({\bf E}_2-{\bf E}_1)\geq \frac 1{2^{k+1}}|{\bf E}_2-{\bf E}_1|^{k+2} \quad \forall k\geq 0.
\end{equation*}
Furthermore, the first inequality in (P3) is used in \cite[eq. (3.4)]{lam2020consistency} for the particular case of $p=2$ and $A=\emptyset$.
Similarly, the second inequality in (P3) and (Q2) is simply a generalization of the standard inequality (see \cite[Sec.12, eq (VII)]{lindqvist2019notes} and \cite[Lem. 2.2]{brasco2013anisotropic}):
\begin{equation*}%    \label{trivial_p>1}
(E_2^{k}\ {\bf E}_2-E_1^{k}\ {\bf E}_1) \cdot ({\bf E}_2-{\bf E}_1)\geq (k+1)(1+ |{\bf E}_2|^2+|{\bf E}_1|^2)^\frac{k}{k+2}|{\bf E}_2-{\bf E}_1|^2 \quad \forall 0>k>-1,
\end{equation*}
where $k=p-2$. Both inequalities in (P3) and (Q2) are mandatory to get a convergence result for the gradient of the scalar potential (see Proposition \ref{lpconv2}). Specifically, (P3) and (Q2) make it possible to \lq\lq code\rq\rq \ the physical requirement that the convergence of the applied boundary potential implies the convergence of the related electrical field and the convergence of the related electrical current densities (see Remark \ref{rem_conv_J}).
\end{rem}

\begin{rem}
Assumptions (R) and (S) model perfect electrical insulators (PEI) and perfect electrical conductors (PEC).

Furthermore, dealing with (R) and (S) means that we are treating a problem in a holed domain, as is easily seen from the associated variational formulations of problems \eqref{minimum_PEI}-\eqref{minimum_PEC}.Optimization problems on domains with holes have received considerable interest in recent years, see e.g. \cite{della2020optimal,gavitone2021isoperimetric,paoli2020stability,paoli2020sharp} and reference therein.
\end{rem}

\subsection{The mathematical model} The electric scalar potential $u$ 
solves the steady current problem \eqref{gproblem}, where $f\in X_\diamond^p(\partial\Omega)$. Problem \eqref{gproblem} is meant in the weak sense, that is
\begin{equation*}
%\label{weak_formulation}
\int_{\Omega }\sigma \left( x,| \nabla u(x) |\right) \nabla u (x) \cdot\nabla \varphi (x)\ \text{d}x=0\quad\forall\varphi\in C_c^\infty(\Omega),
\end{equation*}
and $u$ restricted to $B$ belongs to $W^{1,p}(B)$, whereas $u$ restricted to $A$ belongs to $W^{1,q}(A)$. Moreover, the solution $u$, as a whole, is an element of the larger functional space $W^{1,p}(\Omega)\cup W^{1,q}(\Omega)$. Hereafter, we will exploit (i) if $p\leq q$ then $W^{1,p}(\Omega)\cup W^{1,q}(\Omega)=W^{1,p}(\Omega)$, and (ii) if $p\geq q$ then $W^{1,p}(\Omega)\cup W^{1,q}(\Omega)=W^{1,q}(\Omega)$. The solution $u$ satisfies the boundary condition in the sense that $u-f\in W_0^{1,p}(\Omega)\cup W_0^{1,q}(\Omega)$ and we write $u|_{\partial\Omega}=f$.

The solution $u$ is variationally characterized as
%\begin{equation}\label{minimum}\min\left\{ \mathbb{E}_\sigma\left( u\right)\ :\ u\in W^{1,p}(\Omega)\cup W^{1,q}(\Omega), \ u|_{\partial\Omega}=f\right\}.\end{equation}

%\begin{equation}\label{sigma}\sigma(x,E)=\begin{cases}\sigma_B(x,E)\quad \text{for a.e.} \  x\in B\ \text{and}\ \forall E>0,\\ \sigma_A(x,E)\quad \text{for a.e.} \  x\in A\ \text{and}\ \forall E>0.\end{cases}\end{equation}

\begin{equation}\label{minimum_BA}
\min_{\substack{u\in W^{1,p}(\Omega)\cup  W^{1,q}(\Omega)\\ u=f\ \text{on}\ \partial \Omega}}\mathbb{E}_\sigma
\left(  u \right).
\end{equation}

In (\ref{minimum_BA}), the functional $\mathbb{E}_\sigma\left( u\right)$ is the Dirichlet Energy
\begin{equation}\label{energy_BA}
\mathbb{E}_\sigma
\left(  u \right) = \int_{\Omega} Q_\sigma (x,|\nabla u(x)|)\ \text{d}x= \int_{B} Q_B (x,|\nabla u(x)|)\ \text{d}x+ \int_A Q_A (x,|\nabla u(x)|)\ \text{d}x
\end{equation} 
where $Q_B$ and $Q_A$ are the (nonnegative) Dirichlet Energy densities in $B$ and in $A$, respectively:
\begin{align*}%\label{Bdensity}
& Q_{B} \left( x,E\right)  :=\int_{0}^{E} \sigma_B\left( x,\xi \right)\xi  \text{d}\xi\quad \text{for a.e.}\ x\in B\ \text{and}\ \forall E\geq0,\\
& Q_{A}\left( x,E\right)  :=\int_{0}^{E} \sigma_A\left( x,\xi \right)\xi  \text{d}\xi\quad \text{for a.e.}\ x\in A\ \text{and}\ \forall E\geq 0,%\label{Adensity}
\end{align*}
and $\sigma_B$ and $\sigma_A$ are the restriction of the electrical conductivity $\sigma$ in $B$ and $A$. 

At this point let us stress that when $A=\emptyset$, we fall into the case treated in \cite{corboesposito2021monotonicity} for $p\geq 2$. 

Moreover, if $\sigma_A(x,E)=0$, we have a PEI in $A$. The corresponding problem is
\begin{equation}
\label{problem_PEI}
\begin{cases}
\dive (\sigma_B(x, |\nabla v|)\nabla v)=0 & \text{in } B \\
\sigma_B(x, |\nabla v|)\partial_\nu v=0 & \text{on } \partial A\\
v\in W^{1,p}(B), \ 
v=f & \text{on }\partial\Omega.
\end{cases}
\end{equation}
The solution $u$ is variationally defined as the minimum of the following Dirichlet Energy
\begin{equation}
\label{minimum_PEI}
\min_{\substack{u\in W^{1,p}(B)\\ u=f \text{on }\partial\Omega}}\mathbb E_\sigma (u),
\end{equation}
where
\[
\mathbb{E}_\sigma(u)=\int_{B} Q_{B}(x,|\nabla u(x)|)dx.%=\int_B\int_0^{|\nabla u|}\sigma_B(x, \xi)\xi d\xi.
\]

Similarly, if $\sigma_A(x,E)=+\infty$, we have a PEC in $A$. The corresponding problem is
\begin{equation}
\label{problem_PEC}
\begin{cases}
\dive (\sigma_B(x, |\nabla v|)\nabla v)=0 & \text{in }B \\
|\nabla v|=0 & \text{a.e. in }A\\
\int_{\partial A_i}\sigma(x,|\nabla v(x)|)\partial_\nu v(x)dS=0 & i=1,...,M\\
v\in W^{1,p}(\Omega), \ v=f & \text{on }\partial\Omega.
\end{cases}
\end{equation}
The solution $u$ is variationally defined as the minimum of the following Dirichlet Energy
\begin{equation}
\label{minimum_PEC}\min_{\substack{u\in W^{1,p}(\Omega)\\ u=f \text{ on } \partial\Omega\\ |\nabla u|=0\text{ in } A}}\mathbb E_\sigma (u),
\end{equation}
where
\begin{equation}
    \label{energy_PEC}
\mathbb{E}_\sigma(u)=\int_{B} Q_{B}(x,|\nabla u(x)|)dx.
\end{equation}
Let us stress that, since the solution in \eqref{minimum_PEC} assumes a constant value in $A$, then we can equivalently consider $\mathbb E_\sigma$ equal to \eqref{energy_PEC}  instead of \eqref{minimum_PEC}.

Throughout the paper, we will use the following notation to denote the solutions of the previously described minimum problems.
\begin{defn}
We denote by $u^g$ the unique solution arising from problems \eqref{minimum_BA}, \eqref{minimum_PEI} and \eqref{minimum_PEC}, corresponding to the boundary data $g \in X_\diamond^p (\partial \Omega)$.    
\end{defn}

It is worth noting that (i) if $u^g$ solves problem \eqref{minimum_BA}, then it is in $W^{1,p}(\Omega)\cup W^{1,q}(\Omega)$, (ii) if $u^g$ solves problem \eqref{minimum_PEI}, then it is in $W^{1,p}(B)$ and (iii) if $u^g$ solves problem \eqref{minimum_PEC}, then is in $W^{1,p}(\Omega)$.

The proof of the existence and uniqueness of the solution for \eqref{minimum_BA}, \eqref{minimum_PEI} and \eqref{minimum_PEC} in its variational form, relies on standard methods of the Calculus of Variations, when the Dirichlet Energy Density presents the same growth exponent in any point of the domain $\Omega$. The case treated in this work is nonstandard, because the Dirichlet Energy Density presents different growth exponents in $B$ and $A$ and, hence, we provide a proof in the following.

\begin{thm}
Let $1<p, q<+\infty$,  $f\in X_\diamond^p(\partial\Omega)$ and $\sigma$ satisfies (P1), (P2) and (P3). 
\begin{itemize}
    \item If $A=\emptyset$, then there exists a unique solution of problem \eqref{minimum_BA};
    \item If (Q1)-(Q2) hold, then there exists a unique solution of problem \eqref{minimum_BA};
    \item If (R) hold, then there exists a unique solution of problem \eqref{minimum_PEI};
    \item If (S) hold, then there exists a unique solution of problem \eqref{minimum_PEC}.
\end{itemize}
\end{thm}
\begin{proof}
A simple integration of the l.h.s. of the assumption (P2) provides the coercitivity of $Q_B(x,\cdot)$ a.e. in $B$.

Furthermore, if we consider two parallel vectors $\bf E_1$ and $\bf E_2$ in (P3), we easily deduce that $\sigma_B(x,E_2){E}_2-\sigma_B(x,E_1){E}_1> 0$ for any $E_1,E_2>0$. This means that $E\in\ [0,+\infty)\mapsto \sigma_B(x, E) E$ is strictly increasing for a.e. $x \in B$, that is $E\in\ [0,+\infty)\mapsto Q_B(\cdot,E)$ is strictly convex a.e. in $B$.

When (Q1) holds, \eqref{min_max_estimates} implies that there exist $\underline Q>0$ such that
\[
Q_\sigma(x, E)\ge  \underline{Q}\left(\frac{E}{E_0}\right)^r\quad\text{for a.e. } x\in\Omega \text{ and } \forall E>0.
\]

Hence, the function $Q_\sigma(x,\cdot)$ is coercive a.e. in $\Omega$. 

The assumption (Q2) implies that $E\in\ [0,+\infty)\mapsto Q_A(\cdot, E)$ is strictly convex a.e. in $A$ and hence $E\in\ [0,+\infty)\mapsto Q_\sigma(\cdot,E)$ is strictly convex a.e. in $\Omega$.

Similarly, when (R) or (S) holds, the convexity of $Q_B$ follows from assumption (P3) whereas the coercitivity follows from the left hand side of assumption (P2).

In all four cases, standard direct methods of the calculus of variations \cite[Th. 3.30]{dacorogna2007direct} provide the existence and uniqueness of the solution.
\end{proof}
\begin{rem}
The condition $f\in X_\diamond^p(\partial\Omega)$ is a  necessary assumption because it guarantees the existence of a function in $f+W_{0}^{1,p}(\Omega)$ such that $\mathbb E_\sigma (f)<+\infty$ (see \cite[Th. 3.30]{dacorogna2007direct} for details). 
\end{rem}

\subsection{The DtN operator.} The Dirichlet-to-Neumann (DtN) operator maps the Dirichlet data into the corresponding Neumann data:
\begin{equation*}%\label{DtN}
\Lambda_\sigma   :f\in X^p_\diamond(\partial\Omega)\mapsto \sigma(x, |\nabla u^f|)\ 
\partial_nu^f|_{\partial\Omega} 
\in X^p_\diamond(\partial\Omega)',
\end{equation*}
where $X^p_\diamond(\partial\Omega)'$ is the dual space of $X^p_\diamond(\partial\Omega)$ and $u^f$ is the solution of \eqref{gproblem}. From the physical standpoint, the DtN operator maps the boundary electric scalar potential into the normal component of the electrical current density entering $\partial \Omega$ (see \ref{gOhm}).

In weak form, the DtN operator is
\begin{equation}
\label{w-DtN}
\langle \Lambda_\sigma  \left( f\right) ,\varphi\rangle
=\int_{\partial \Omega }\varphi (x) {\sigma}\left( x, \left\vert \nabla u^f(x)\right\vert\right)  \partial_{{{n}}} u^f(x)\text{d}S\quad\forall \varphi\in X^p_\diamond(\partial\Omega).
\end{equation}

Furthermore, by testing the DtN operator \eqref{w-DtN} with the solution $u^f$ of \eqref{gproblem} and using a divergence Theorem, we obtain the ohmic power dissipated by the conducting material:
\begin{equation*}%\label{w-DtN-f}
\langle \Lambda_\sigma  \left( f\right) ,f\rangle%=\int_{\Omega }\sigma \left( x\right) |\nabla u (x)|^2\ \text{d}x
=\int_{\Omega } {\sigma}( x ,\nabla u^f(x)){ |\nabla u^f}(x)|^2\ \text{d}x.
\end{equation*}
If $\varphi\neq f$ in \eqref{w-DtN}, we have the so-called
\emph{virtual power} product that plays an important role since it is equal to the G\^{a}teaux derivative of the Dirichlet Energy, when evaluated at the solution $u^f$, as shown in Section \ref{connection_sec}.

The injectivity of the DtN operator is guaranteed by the assumption of zero average of $f$.

\section{Convergence of the Scalar Potential w.r.t. the Boundary Data}\label{convergence_sec}
The aim of this section is to prove a fundamental convergence result of the solutions with respect to the boundary data. Specifically, it is possible to prove that the operator $f \mapsto \nabla u^f$ is continuous with respect to the boundary data $f$, along any arbitrary direction (see Proposition \ref{lpconv2}). This generalizes the result of \cite[Lem. 3.1]{corboesposito2021monotonicity}, which proved a convergence result for $p\ge 2$ and $q=p$. Here we generalize these previous results to arbitrary $p$ and $q$ with $1 < p,q < +\infty$.

Firstly, we prove the following Lemma that provides an upper bound to the norm of the solutions of problems \eqref{minimum_BA}, \eqref{minimum_PEI}, \eqref{minimum_PEC}, through the norm of the boundary data. The proof, based on well-established methods, is given explicitly because it refers to an atypical configuration.

\begin{lem}\label{X_trace_inequality_lem}
Let $1<p,q<+\infty$, $f\in X_\diamond^p(\partial\Omega)$, ${\sigma}$ satisfying (P1), (P2), (P3). Then, there exists $C=C(p,q, \Omega)$ such that
\begin{itemize}
\item[(i)] If $A=\emptyset$, then
\begin{equation}
\label{lemma_traccia_Omega}
\begin{split}
||\nabla u^f||_{L^p(\Omega)}\leq C(1+ ||f||_{X^p(\partial\Omega)}),
\end{split}
\end{equation}
where $u^{f}\in W^{1,p}(\Omega)$ is the minimizer of \eqref{minimum_BA}.
\item[(ii)] If ${\sigma}$ satisfies (QX), then
\begin{equation}
\label{lemma_traccia_AB}
\begin{split}
||\nabla u^f||_{L^p(B)},||\nabla u^f||^\frac pq_{L^q(A)}&\leq\left(||\nabla u^f||^p_{L^p(B)}+||\nabla u^f||^q_{L^q(A)}\right)^\frac 1p\\
&\leq C(1+ ||f||_{X^p(\partial\Omega)}),
\end{split}
\end{equation}
where $u^{f}\in W^{1,p}(\Omega)$ is the minimizer of \eqref{minimum_BA}.
 \item[(iii)] If ${\sigma}$ satisfies  (R), then
\begin{equation}
\label{lemma_traccia_PEI}
||\nabla u^f||_{L^p(B)}\leq C(1+ ||f||_{X^p(\partial\Omega)}),
\end{equation}
where $u^{f}\in W^{1,p}(\Omega)$ is the minimizer of \eqref{minimum_PEI}.
 \item[(iv)] If ${\sigma}$ satisfies (S), then
\begin{equation}\label{lemma_traccia_PEC}
||\nabla u^f||_{L^p(B)}\leq C(1+ ||f||_{X^p(\partial\Omega)}),
\end{equation}
where $u^{f}\in W^{1,p}(\Omega)$ is the minimizer of \eqref{minimum_PEC}.
\end{itemize}
\end{lem}
\begin{proof}
We give the proof for each individual case.

$\bullet$ \textbf{Claim (i)}.
Let us assume $p \geq 2$. Since $u^f$ realizes the minimum in \eqref{minimum_BA} with $A=\emptyset$, by using (P2), we have
\begin{equation}
    \label{chain_p_en}
\underline c\int_\Omega |\nabla u^f(x)|^p \text{d}x\leq\int_\Omega Q_\sigma(x,|\nabla{u^f}(x)|)\text{d}x\le \overline c \int_\Omega\left(|\nabla u(x)|^2+ |\nabla u(x)|^p\right)\text{d}x\end{equation}
for any $u\in f+W^{1,p}_0(\Omega)$, and where the two constants $\overline c$ and $\underline c$ are defined as
\begin{equation}
\label{consts}
    \overline c =\frac{\overline{\sigma}}{pE_0^{p-2}}, \ \underline c =\frac{\underline{\sigma}}{pE_0^{p-2}}.
\end{equation}
%By applying H\"older inequality, we have
%\begin{equation}\label{after_H}
%c_1\int_\Omega |\nabla u^f(x)|^p \text{d}x\leq\int_\Omega Q(x,|\nabla{u^f}(x)|)\text{d}x\le\max\left\{\frac{ c(|\Omega|)}{||\nabla u||_p^{p-2}},c_2\right\} \int_\Omega |\nabla u(x)|^p\text{d}x
%\end{equation}
%for any $u\in f+W^{1,p}_0(\Omega)$, where $c(|\Omega|)$ is a positive constant depending on the measure of $\Omega$.
%\begin{equation}    \label{pEnergy}\min_{w^f\in f+W^{1,p}_0(\Omega)}\int_\Omega |\nabla w^f(x)|^p\ \text{d}x.\end{equation}

The Inverse Trace inequality in Besov spaces \cite[Th. 18.34]{leoni17} assures the existence of a function $w\in f+ W^{1,p}_0(\Omega)$ such that $||\nabla w||_{p}\leq K(p,\Omega) ||f||_{X^p(\partial\Omega)}$.
Therefore, we have
\begin{equation*}
\begin{split}
||\nabla u^f||^p_{p}\leq C_1(||\nabla w||^2_{2}+||\nabla w||^p_{p})\le C_1(C_2||\nabla w||^2_{p}+||\nabla w||^p_{p})\\
\le C_1(C_2+(C_2+1)||\nabla w||^p_{p})\le C_1(C_2+(C_2+1) K^p_p||f||^p_{X^p(\partial\Omega)}).
\end{split}
\end{equation*}
where in the first inequality we have used \eqref{chain_p_en} with $u=w$ and $C_1=\overline c / \underline c$; in the second inequality that, for $p \ge 2$ there exists a constant $C$ such that $||\nabla w||_2\le C_2 ||\nabla w ||_p$ by the continuous embedding of the Sobolev spaces; in the third inequality that $||\nabla w||_p^2\le 1+ ||\nabla w ||_p^p$ and in the fourth inequality the inverse trace inequality. Therefore, \eqref{lemma_traccia_Omega} follows by setting $C=C_1\max\{C_2,(C_2+1)K_p^p\}$.

The inequality for $1<p<2$ analogously follows (in easier way) using the second line of the assumption (P2). 

$\bullet$ \textbf{Claim (ii)}.
We have
\[
\begin{split}
\underline c\left(\int_B |\nabla u^f(x)|^p dx+\int_A |\nabla u^f(x)|^q dx \right)&\leq \mathbb E_\sigma (u^f)\\ 
&\hspace{-5cm}\leq \mathbb E_\sigma (u_\infty^f)= \int_BQ_B(x,|\nabla{u^f_\infty}(x)|)\text{d}x\leq \overline c C\left(1+ ||f||^p_{X^p(\partial\Omega)}\right),
\end{split}
\]
where $\overline c$ and $\underline c$ are the constants defined in \eqref{consts} and $u_\infty^f$ is the solution of the PEC problem
\eqref{minimum_PEC}, the first inequality follows from assumption (P2)-left and (Q1)-left, the second inequality follows from using $u_\infty^f$ as the test function, and the third inequality follows by using the argument of the previous case.

$\bullet$ \textbf{Claim (iii)}. 
Regarding the case where (R) holds, that is the PEI case, let us consider the solution $\hat u^f$ of the problem  $\min_{\substack{u\in W^{1,p}(\Omega)\\ u=f\ \text{on}\ \partial \Omega}}\int_\Omega \hat Q(x,|\nabla{u}(x)|)\text{d}x$, where
\[
\hat Q(x,E)=
\begin{cases}
 Q_B(x,E)\ &\text{on}\ B,\\
 E^p\ &\text{on}\ A.
\end{cases}
\]
It is easy to see that there exist two positive constants $\underline c$ and $\overline c$ such that
\[
\begin{split}
\underline c\int_B |\nabla u^f(x)|^p dx & \leq \mathbb E_\sigma (u^f)\leq \mathbb E_\sigma (\hat u^f)\leq \overline c\left( \int_B |\nabla \hat u^f(x)|^p\text{d}x+\int_A |\nabla \hat u^f(x)|^p\text{d}x\right)\\
&\leq C ||f||_{X^p(\partial\Omega)}.
\end{split}
\]
where in the first inequality we have used the l.h.s. of (P2), in the second inequality we have $\hat u^f$ as the admissible test function, in the third inequality we have used the r.h.s of assumption (P2) and the last inequality follows as in Claim {\it (i)}.

$\bullet$ \textbf{Claim (iv)}. 
The conclusion for the PEC case follows, arguing as in Claim {\it (i)}.
\end{proof}

Lemma \ref{X_trace_inequality_lem} underpins the following fundamental convergence result. 
\begin{prop}
\label{lpconv2}
Let $1<p,q<+\infty$, $f\in X_\diamond^p(\partial\Omega)$ and ${\sigma}$ satisfying (P1), (P2), (P3).
\begin{itemize}
\item[(i)] If $A=\emptyset$, then \begin{equation}
    \label{fund_conv_Omega}
\begin{split}
& \nabla u^{f+\varepsilon \varphi}\to\nabla u^f\ \text{in} \ {L^p(\Omega)}\qquad\text{as}\ {\varepsilon\to 0^+},
\end{split}
\end{equation}
for any $\varphi\in W^{1,p}(\Omega)$, where $u^{f+\varepsilon \varphi} \in W^{1,p}(\Omega)$ (resp. $u^{f} \in W^{1,p}(\Omega)$) is the minimizer of \eqref{minimum_BA} with boundary data $f+\varepsilon \varphi$ (resp. $f$).

\item[(ii)] If ${\sigma}$ satisfies (Q1)-(Q2), then \begin{equation}
    \label{fund_conv_AB}
\begin{split}
& \nabla u^{f+\varepsilon \varphi}\to\nabla u^f\ \text{in} \ {L^p(B)}\qquad \text{as}\ {\varepsilon\to 0^+},\\
& \nabla u^{f+\varepsilon \varphi}\to\nabla u^f\ \text{in} \ {L^q(A)}\qquad \text{as}\ {\varepsilon\to 0^+},\\
& \nabla u^{f+\varepsilon \varphi}\to\nabla u^f\ \text{in} \ {L^r(\Omega)}\qquad\text{as}\ {\varepsilon\to 0^+},
\end{split}
\end{equation}
for any $\varphi\in (W^{1,p}(\Omega)\cup W^{1,q}(\Omega))\cap W^{1,p}(B)\cap W^{1,q}(A)$, where $r=\min\{p,q\}$ and $u^{f+\varepsilon \varphi} \in W^{1,p}(\Omega)$ (resp. $u^{f} \in W^{1,p}(\Omega)$) is the minimizer of \eqref{minimum_BA} with boundary data $f+\varepsilon \varphi$ (resp. $f$).
\item[(iii)] If ${\sigma}$ satisfies (R), then
\begin{equation}
    \label{fund_conv_PEI}
\nabla u^{f+\varepsilon \varphi}\to\nabla u^f\ \text{in} \ {L^p(B)}\qquad \text{as}\ {\varepsilon\to 0^+},
\end{equation}
where $u^{f+\varepsilon \varphi} \in W^{1,p}(\Omega)$ (resp. $u^{f} \in W^{1,p}(\Omega)$) is the minimizer of \eqref{minimum_PEI} with boundary data $f+\varepsilon \varphi$ (resp. $f$).
\item[(iv)] If ${\sigma}$ satisfies (S), then
\begin{equation}
    \label{fund_conv_PEC}
\nabla u^{f+\varepsilon \varphi}\to\nabla u^f\ \text{in} \ {L^p(B)}\qquad \text{as}\ {\varepsilon\to 0^+},
\end{equation}
where $u^{f+\varepsilon \varphi} \in W^{1,p}(\Omega)$ (resp. $u^{f} \in W^{1,p}(\Omega)$) is the minimizer of \eqref{minimum_PEC} with boundary data $f+\varepsilon \varphi$ (resp. $f$).
\end{itemize}
\end{prop}
\begin{proof}
For any fixed $0<\varepsilon<1$, we denote by $I$ and $II$ the following quantities:
\begin{align}\label{Idiff}
I:&=\int_\Omega  {\sigma}\left( x, \left\vert \nabla u^{f+\varepsilon \varphi}(x)\right\vert\right)  {\nabla u^{f+\varepsilon \varphi}(x)}\cdot(  \nabla u^{f+\varepsilon \varphi}(x)- \nabla u^{f}(x))\text{d}x,\\
\label{IIdiff}
II:&=\int_\Omega  {\sigma}\left( x, \left\vert \nabla u^{f}(x)\right\vert\right)  {\nabla u^{f}(x)}\cdot(  \nabla u^{f+\varepsilon \varphi}(x)- \nabla u^{f}(x))\text{d}x.
\end{align}

$\bullet$ {\bf Claim (i)}. We distinguish the proof according to the value of the exponent $p$.

For $p\geq 2$, %we refer to \cite[Lem.3.1]{corboesposito2021monotonicity}, but for sake of completeness, we also give the proof of this case.
the divergence Theorem gives
\begin{align}
\label{Ibordo}
I&=\varepsilon\int_\Omega {\sigma}\left( x, \left\vert \nabla u^{f+\varepsilon \varphi}(x)\right\vert\right) \nabla u^{f+\varepsilon \varphi}(x)\cdot \nabla \varphi(x)\text{d}x\\%&=\varepsilon \int_{\partial\Omega} {\sigma}\left( x, \left\vert \nabla u^{f+\varepsilon g}(x)\right\vert\right) \partial_\nu u^{f+\varepsilon g}(x) g(x)\text{d}S,\\
\label{IIbordo}
II&=\varepsilon\int_{ \Omega}  {\sigma}\left( x, \left\vert \nabla u^{f}(x)\right\vert\right)  \nabla u^{f}(x)\cdot \nabla \varphi(x)\text{d}x.
\end{align}
By subtracting \eqref{IIbordo} from \eqref{Ibordo}, we have
\begin{equation*}%\label{magg_diff_Omega}
\begin{split}
I-II & = \varepsilon\int_{\Omega} \nabla \varphi(x)\left[  {\sigma}\left( x, \left\vert \nabla u^{f+\varepsilon \varphi}(x)\right\vert\right)  \nabla u^{f+\varepsilon \varphi}(x)- {\sigma}\left( x, \left\vert \nabla u^{f}(x)\right\vert\right)  \nabla u^{f}(x)\right]\ \text{d}x\\
& \leq\varepsilon D_1\int_{\Omega}| \nabla \varphi(x)| \left[ \left\vert \nabla u^{f+\varepsilon \varphi}(x)\right\vert +\left\vert \nabla u^{f+\varepsilon \varphi}(x)\right\vert^{p-1} + \left\vert \nabla u^{f}(x)\right|+\left\vert \nabla u^{f}(x)\right\vert^{p-1} \right]\ \text{d}x,\\
\end{split}
\end{equation*}
where in the second line we have used (P2)-right with $D_1=\overline\sigma\max\left\{1, {E_0^{2-p}}\right\}$. Then, we have
\begin{equation}\label{maggiorazione_Omega}
\begin{split}
I-II &\leq \varepsilon D_1  \left[\left|\left|\nabla \varphi\right|\right|_{L^2(\Omega)} \left\vert\left\vert \nabla u^{f+\varepsilon \varphi}\right\vert\right\vert_{L^2(\Omega)}+\left|\left|\nabla \varphi\right|\right|_{L^p(\Omega)} \left\vert\left\vert \nabla u^{f+\varepsilon \varphi}\right\vert\right\vert_{L^p(\Omega)}^{p-1} \right.\\
&\qquad\qquad\qquad\qquad\quad\left. +\left|\left|\nabla \varphi\right|\right|_{L^2(\Omega)} \left\vert\left\vert \nabla u^{f}\right\vert\right\vert_{L^2(\Omega)}+\left|\left|\nabla \varphi\right|\right|_{L^p(\Omega)} \left\vert\left\vert \nabla u^{f}\right\vert\right\vert_{L^p(\Omega)}^{p-1}\right]\\
&\leq \varepsilon CD_1\left[\left|\left|\nabla \varphi\right|\right|_{L^2(\Omega)} (1+\left\vert\left\vert  f+\varepsilon \varphi\right\vert\right\vert_{X^2(\partial\Omega)})+ \left|\left|\nabla \varphi\right|\right|_{L^p(\Omega)}(1+||f+\varepsilon \varphi||_{X^p(\partial\Omega)} )^{p-1}\right.\\
&\qquad\qquad\qquad\qquad\quad\left.+\left|\left|\nabla \varphi\right|\right|_{L^2(\Omega)}(1+ \left\vert\left\vert f\right\vert\right\vert_{X^2(\partial\Omega)})+ \left|\left|\nabla \varphi \right|\right|_{L^p(\Omega)}(1+||f||_{X^p(\partial\Omega)})^{p-1}  \right]\\
&\leq \varepsilon CD_1\bigg[\left|\left|\nabla \varphi\right|\right|_{L^2(\Omega)} \left(1+ \left\vert\left\vert f\right|\right|_{X^2(\partial\Omega)}+ \left\vert\left\vert \varphi\right\vert\right\vert_{X^2(\partial\Omega)}\right)\\
&\qquad\qquad\qquad\qquad\quad+\left. \left|\left|\nabla \varphi\right|\right|_{L^p(\Omega)} \left(1+||f||_{X^p(\partial\Omega)}+||\varphi||_{X^p(\partial\Omega)}\right)^{p-1} \right],
\end{split}
\end{equation}
where in the first inequality we used the H\"older inequality, in the second inequality we used \eqref{lemma_traccia_Omega} from Lemma \ref{X_trace_inequality_lem}, in the third inequality we used the elementary properties of the norm and the fact that $\varepsilon<1$.

The quantity in the squared bracket of the last term in \eqref{maggiorazione_Omega} is finite because $f,\varphi\in X^p_\diamond(\partial\Omega)$. Hence, $I-II$ is upper bounded:
\begin{equation}
    \label{uppconv_Omega}
    I-II\leq  \varepsilon F,
\end{equation}
where $F$ is the product of $C$, $D_1$ and the quantity in the squared brackets on the r.h.s. of \eqref{maggiorazione_Omega}.

On the other hand, by using (P3), we have:
\begin{equation}
\label{lowconv_Omega}
||\nabla u^{f+\varepsilon  g }-\nabla u^f||^p_{L^p(\Omega)}\leq \frac{1}{\kappa}(I-II).
\end{equation}

By joining \eqref{uppconv_Omega} and \eqref{lowconv_Omega}, we have
\begin{equation*}%\label{convconv}
||\nabla u^{f+\varepsilon  g }-\nabla u^f||^p_{L^p(\Omega)}\leq \frac 1 \kappa( I-II)\leq \frac{F}{\kappa}\varepsilon.
\end{equation*} 
This gives \eqref{fund_conv_Omega} by passing the to limit as $\varepsilon\to 0^+$.

If $1<p<2$, the inequality \eqref{uppconv_Omega} follows analogously but with slightly different constants. 

Regarding the lower bound, we have
\begin{equation}\label{lowconv_Omega<2}
\begin{split}
&||\nabla u^{f+\varepsilon  \varphi}-\nabla u^f||^p_{p}\\
&\le  \frac{1}{\kappa}\int_\Omega (1+ |\nabla u^{f+\varepsilon  g}(x)|^2+|\nabla u^f(x)|^2)^\frac{2-p}4\\
&\qquad\qquad ((\sigma(x,|\nabla u^{f+\varepsilon  \varphi}(x)|)-\sigma(x,|\nabla u^{f}(x)|))(\nabla u^{f+\varepsilon  \varphi}(x)-\nabla u^{f}(x)))^\frac p2 \text{d}x\\
&\leq \frac{1}{\kappa}\left(\int_\Omega (1+ |\nabla u^{f+\varepsilon  \varphi}(x)|^2+|\nabla u^f(x)|^2)^\frac 12\text{d}x\right)^\frac{2-p}2\\
&\qquad\qquad\left(\int_\Omega(\sigma(x,|\nabla u^{f+\varepsilon  \varphi}(x)|)-\sigma(x,|\nabla u^{f}(x)|))(\nabla u^{f+\varepsilon  \varphi}(x)-\nabla u^{f}(x)) \text{d}x\right)^\frac p2\\
&= \frac{G}{\kappa} (I-II)^\frac p2,
\end{split}
\end{equation}
where in the first inequality we used the assumption (P3) raised to the power $p/2$, in the second inequality we used the H\"older inequality with exponents $2/p$ and $2/(2-p)$, and $G=\left(\int_\Omega (1+ |\nabla u^{f+\varepsilon  \varphi}(x)|^2+|\nabla u^f(x)|^2)^\frac 12\text{d}x\right)^\frac{2-p}2$ is a finite quantity.

By taking into account \eqref{uppconv_Omega} and \eqref{lowconv_Omega<2}, we have
\begin{equation*}
%\label{convconv}
||\nabla u^{f+\varepsilon  g }-\nabla u^f||^p_{L^p(\Omega)}\leq \frac{G}{\kappa}( I-II)^\frac p2 \leq \frac{G F^\frac p2}{\kappa}\varepsilon^\frac p2,
\end{equation*}
which gives \eqref{fund_conv_Omega} by passing the to limit as $\varepsilon\to 0^+$.

$\bullet$ {\bf Claim (ii)}. First, we notice that the divergence Theorem gives \eqref{Ibordo} and \eqref{IIbordo}, also in this case.
For $p,q\ge 2$, we have
\begin{equation*}%\label{magg_diff_AB}
\begin{split}
I-II & = \varepsilon\int_{\Omega} \nabla \varphi(x)\left[  {\sigma}\left( x, \left\vert \nabla u^{f+\varepsilon \varphi}(x)\right\vert\right)  \nabla u^{f+\varepsilon \varphi}(x)- {\sigma}\left( x, \left\vert \nabla u^{f}(x)\right\vert\right)  \nabla u^{f}(x)\right]\ \text{d}x\\
& \leq\varepsilon D_1\int_{B}| \nabla \varphi(x)| \left[ \left\vert \nabla u^{f+\varepsilon \varphi}(x)\right\vert +\left\vert \nabla u^{f+\varepsilon \varphi}(x)\right\vert^{p-1} + \left\vert \nabla u^{f}(x)\right|+\left\vert \nabla u^{f}(x)\right\vert^{p-1} \right]\ \text{d}x\\
&\ \  +\varepsilon D_1\int_{A}| \nabla \varphi(x)| \left[ \left\vert \nabla u^{f+\varepsilon \varphi}(x)\right\vert +\left\vert \nabla u^{f+\varepsilon \varphi}(x)\right\vert^{q-1} + \left\vert \nabla u^{f}(x)\right|+\left\vert \nabla u^{f}(x)\right\vert^{q-1} \right]\ \text{d}x,
\end{split}
\end{equation*}
where in the first inequality we used (P2) and (Q1).
Then, we have
\begin{equation}\label{maggiorazione_AB}
\begin{split}
I-II &\leq \varepsilon D_1  \left[\left|\left|\nabla \varphi\right|\right|_{L^2(B)} \left\vert\left\vert \nabla u^{f+\varepsilon \varphi}\right\vert\right\vert_{L^2(B)}+\left|\left|\nabla \varphi\right|\right|_{L^p(B)} \left\vert\left\vert \nabla u^{f+\varepsilon \varphi}\right\vert\right\vert_{L^p(B)}^{p-1} \right.\\
& \qquad\qquad +\left|\left|\nabla \varphi\right|\right|_{L^2(B)} \left\vert\left\vert \nabla u^{f}\right\vert\right\vert_{L^2(B)}+\left|\left|\nabla \varphi\right|\right|_{L^p(B)} \left\vert\left\vert \nabla u^{f}\right\vert\right\vert_{L^p(B)}^{p-1}\\
&\qquad\qquad\qquad  +\left|\left|\nabla \varphi\right|\right|_{L^2(A)} \left\vert\left\vert \nabla u^{f+\varepsilon \varphi}\right\vert\right\vert_{L^2(A)}+\left|\left|\nabla \varphi\right|\right|_{L^q(A)}\left\vert\left\vert \nabla u^{f+\varepsilon \varphi}\right\vert\right\vert_{L^q(A)}^{q-1}\\
& \qquad\qquad\qquad\qquad\ +\left. \left|\left|\nabla \varphi \right|\right|_{L^2(A)} \left\vert\left\vert \nabla u^{f}\right\vert\right\vert_{L^2(A)} + \left|\left|\nabla \varphi \right|\right|_{L^q(A)} \left\vert\left\vert \nabla u^{f}\right\vert\right\vert_{L^q(A)}^{q-1}\right]\\
&\leq \varepsilon C D_1\left[\left|\left|\nabla \varphi\right|\right|_{L^2(B)} \left\vert\left\vert  f+\varepsilon \varphi\right\vert\right\vert_{X^2(\partial\Omega)}+ \left|\left|\nabla \varphi\right|\right|_{L^p(B)}||f+\varepsilon \varphi||^{p-1}_{X^p(\partial\Omega)} \right.\\
&\qquad\qquad+\left|\left|\nabla \varphi\right|\right|_{L^2(B)} \left\vert\left\vert f\right\vert\right\vert_{X^2(\partial\Omega)}+ \left|\left|\nabla \varphi \right|\right|_{L^p(B)}||f||^{p-1}_{X^p(\partial\Omega)} \\
&\qquad\qquad+ \left|\left|\nabla \varphi\right|\right|_{L^2(A)} \left\vert\left\vert  f+\varepsilon \varphi\right\vert\right\vert_{X^2(\partial\Omega)}^\frac p2+ \left|\left|\nabla \varphi\right|\right|_{L^q(A)}||f+\varepsilon \varphi||^{\frac{p(q-1)}{q}}_{X^p(\partial\Omega)} \\
&\qquad\qquad\left.+\left|\left|\nabla \varphi\right|\right|_{L^2(A)} \left\vert\left\vert f\right\vert\right\vert_{X^2(\partial\Omega)}^\frac p2+ \left|\left|\nabla \varphi \right|\right|_{L^q(A)}||f||^{p-1}_{X^p(\partial\Omega)}+||f||^{\frac {p(q-1)}q}_{X^p(\partial\Omega)} \right]\\
&\leq \varepsilon C D_1\bigg[\left|\left|\nabla \varphi\right|\right|_{L^2(\Omega)} \left(\left\vert\left\vert f\right|\right|_{X^2(\partial\Omega)}+ \left\vert\left\vert \varphi\right\vert\right\vert_{X^2(\partial\Omega)}\right)\\
&\qquad\qquad+\left. \left|\left|\nabla \varphi\right|\right|_{L^p(\Omega)} \left(||f||^{p-1}_{X^p(\partial\Omega)}+||\varphi||^{p-1}_{X^p(\partial\Omega)}+ ||f||^{\frac {p(q-1)}q}_{X^p(\partial\Omega)}+||\varphi||^{\frac {p(q-1)}q}_{X^p(\partial\Omega)}\right) \right],
\end{split}
\end{equation}
where in the first inequality we used the H\"older inequality, in the second inequality we used \eqref{lemma_traccia_AB} from Lemma \ref{X_trace_inequality_lem}, in the third inequality we used the elementary properties of the norm and the fact that $\varepsilon<1$.

The quantity in the squared bracket of the last term in \eqref{maggiorazione_AB} is finite because $f,\varphi\in X^p_\diamond(\partial\Omega)$. Hence, the quantity $I-II$ is upper bounded:
\begin{equation}
    \label{uppconv_AB}
    I-II\leq F \varepsilon,
\end{equation}
where $F$ is the product of $C_1$, $D$ and the last term in the squared parenthesis of \eqref{maggiorazione_AB}.

In the remaining cases ($1<p<2$ and/or $1<q<2$), the inequality \eqref{uppconv_AB} follows analogously, but with slightly different constants. 

The lower bound for $I-II$ depends on the values of $p$ and $q$. Specifically, for $p\ge 2$, by using (P3), we have
\begin{equation}\label{lowconv_AB_p>2}
||\nabla u^{f+\varepsilon  g }-\nabla u^f||^p_{L^p(B)}\leq \frac 1 \kappa ( I-II),
\end{equation}
whereas for $1<p<2$, we have
\begin{equation}\label{lowconv_AB_p<2}
\begin{split}
&||\nabla u^{f+\varepsilon  \varphi}-\nabla u^f||^p_{L^p(B)}\\
&\le  \frac 1 {\kappa}\int_B [1+ |\nabla u^{f+\varepsilon  g}(x)|^2+|\nabla u^f(x)|^2]^\frac{2-p}4\\
&\qquad\qquad ((\sigma(x,|\nabla u^{f+\varepsilon  \varphi}(x)|)-\sigma(x,|\nabla u^{f}(x)|))(\nabla u^{f+\varepsilon  \varphi}(x)-\nabla u^{f}(x)))^\frac p2 \text{d}x\\
&\leq \frac 1 {\kappa}\left( \int_B (1+ |\nabla u^{f+\varepsilon  \varphi}(x)|^2+|\nabla u^f(x)|^2)^\frac 12\text{d}x\right)^\frac{2-p}2\\
&\qquad\qquad\left(\int_B(\sigma(x,|\nabla u^{f+\varepsilon  \varphi}(x)|)-\sigma(x,|\nabla u^{f}(x)|))(\nabla u^{f+\varepsilon  \varphi}(x)-\nabla u^{f}(x)) \text{d}x\right)^\frac p2\\
&\leq \frac G {\kappa} (I-II)^\frac p2,
\end{split}
\end{equation}
where in the first inequality we used the assumption (P3) raised to the power $p/2$, in the second inequality we used the H\"older inequality with exponents $2/p$ and $2/(2-p)$.

Regarding the inner convergence result, for $q\ge 2$ and by using (Q2), we have
\begin{equation}\label{lowconv_AB_q>2}
||\nabla u^{f+\varepsilon  g }-\nabla u^f||^q_{L^q(A)}\leq \frac 1 \kappa ( I-II),
\end{equation}
whereas for $1<q<2$ we have 
\begin{equation}\label{lowconv_AB_q<2}
||\nabla u^{f+\varepsilon  \varphi}-\nabla u^f||^q_{L^q(A)}\leq\frac G {\kappa} (I-II)^\frac q2,
\end{equation}
as follows by using (Q2) instead of (P3) and by substituting $q$ with $p$ and $A$ with $B$ in \eqref{lowconv_AB_p<2}.

By combining \eqref{uppconv_AB} with either \eqref{lowconv_AB_p>2} or \eqref{lowconv_AB_p<2}, and by letting $\varepsilon\to 0^+$, we have the outer convergence result of \eqref{fund_conv_AB}. Similarly, by combining \eqref{uppconv_AB} with either  \eqref{lowconv_AB_q>2} or \eqref{lowconv_AB_q<2}, and by letting $\varepsilon\to 0^+$, we have the inner convergence result of \eqref{fund_conv_AB}.

Finally, the convergence result over the whole domain $\Omega$ (see \eqref{fund_conv_AB}) follows by observing that $L^r(\Omega)= L^p(\Omega)\cup L^q(\Omega)$.

$\bullet$ {\bf Claim (iii)}. From the divergence theorem, and taking into account that in the PEI case the solutions have vanishing normal derivative on $\partial A$ (see \eqref{problem_PEI}), we have
%\[\begin{split}&\int_B\sigma_B(x, |\nabla u^f(x)|)\nabla u^f(x) \nabla \phi(x) dx\\&=\int_B \dive(\sigma_B(x, |\nabla u^f(x)|)\nabla u^f(x)\ \phi(x))dx \\&=\int_{\partial\Omega}\sigma_B(x, |\nabla u^f(x)|)\partial_\nu u^f(x)\ \phi(x)dS\\&=\int_{\partial\Omega}\sigma_B(x, |\nabla u^f(x)|)\partial_\nu u^f(x)\ f(x) dS\\&=\int_B\dive(\sigma_B(x, |\nabla u^f(x)|)\nabla u^f(x)\ f(x))\\&=\int_B\sigma_B (x, |\nabla u^f(x) |)\nabla u^f(x)\nabla f(x) dx.\end{split}\]for any $\phi\in \ f+ W^{1,p}_0(\Omega)$. Therefore the integrals in \eqref{Idiff}-\eqref{IIdiff} are equal to the integrals on \eqref{Ibordo}-\eqref{IIbordo} on $B$, 

\begin{align*}%\label{Ibordo_PEI}
I&=\varepsilon\int_B {\sigma}\left( x, \left\vert \nabla u^{f+\varepsilon \varphi}(x)\right\vert\right) \nabla u^{f+\varepsilon \varphi}(x)\cdot \nabla \varphi(x)\text{d}x\\%&=\varepsilon \int_{\partial\Omega} {\sigma}\left( x, \left\vert \nabla u^{f+\varepsilon g}(x)\right\vert\right) \partial_\nu u^{f+\varepsilon g}(x) g(x)\text{d}S,\\
%\label{IIbordo_PEI}
II&=\varepsilon\int_{B}  {\sigma}\left( x, \left\vert \nabla u^{f}(x)\right\vert\right)  \nabla u^{f}(x)\cdot \nabla \varphi(x)\text{d}x.
\end{align*}
The conclusion \eqref{fund_conv_PEI} follows by arguing as in the previous case, but using \eqref{lemma_traccia_PEI} from Lemma \ref{X_trace_inequality_lem}.

$\bullet$ {\bf Claim (iv)}. %From the weak formulation \eqref{weak_formulation}, and by choosing test functions that are constant in $A$, it results that\[\int_B \sigma_B (x, |\nabla u^{f+\varepsilon\varphi}(x)|) \nabla u^{f+\varepsilon\varphi}(x) \cdot \nabla\varphi(x)dx=0\quad\forall\varphi\in C_c^\infty(\Omega),\ \varphi|_{A}=const.\]
%From the third constraint appearing in \eqref{problem_PEC}, it results that
%In this case, since the solution $u^f$ individuates the optimal constant on $A$, this forces the average flux on $A$ to be zero:\[\int_{\partial A} \varphi\sigma_B (x, |\nabla u^{f+\varepsilon\varphi}(x)|) \partial_\nu u^{f+\varepsilon\varphi}(x)dx=0\quad\forall\varphi\in C_c^\infty(\Omega),\ \varphi|_{A}=const.\]
%Analogously, it results that\[\int_{\partial A} \varphi\sigma_B (x, |\nabla u^{f}(x)|) \partial_\nu u^{f}(x)dx=0\quad\forall\varphi\in C_c^\infty(\Omega),\ \varphi|_{A}=const.\]
First, we have
\[
\begin{split}
I&=\varepsilon\int_B {\sigma}\left( x, \left\vert \nabla u^{f+\varepsilon \varphi}(x)\right\vert\right) \nabla u^{f+\varepsilon \varphi}(x)\cdot \nabla \varphi(x)\text{d}x\\
II&=\varepsilon\int_{B}  {\sigma}\left( x, \left\vert \nabla u^{f}(x)\right\vert\right)  \nabla u^{f}(x)\cdot \nabla \varphi(x)\text{d}x.
\end{split}
\]
Indeed, from the definition of $I$ appearing in \eqref{Idiff}, it follows that
\[
\begin{split}
I%&=\int_B\sigma_B(x, |\nabla u^{f+\varepsilon\varphi}(x)|)\nabla u^{f+\varepsilon\varphi}(x) \cdot (\nabla u^{f+\varepsilon \varphi}(x)-\nabla u^{f}(x)) dx\\
&=\int_B \dive(\sigma_B(x, |\nabla u^{f+\varepsilon\varphi}|)\nabla u^{f+\varepsilon\varphi}(x)\ (u^{f+\varepsilon \varphi}(x)-u^{f}(x)))dx \\
&=\varepsilon \int_{\partial\Omega}\sigma_B(x, |\nabla u^{f+\varepsilon\varphi}(x)|)\partial_\nu u^{f+\varepsilon\varphi}(x)\ \varphi(x)|_\Omega dS\\
&\qquad + \sum_{i=1}^M (c^i_\varepsilon-c^i_0)\int_{\partial A}\sigma_B(x, |\nabla u^{f+\varepsilon\varphi}(x)|)\partial_\nu u^{f+\varepsilon\varphi}(x)\  dS\\
%&=\int_{\partial A}\sigma_A(x, |\nabla u^f(x)|)\partial_\nu u^f(x)\ c_f \ dS\\
&=\varepsilon \int_B\dive(\sigma_B(x, |\nabla u^{f+\varepsilon\varphi}(x)|)\nabla u^{f+\varepsilon\varphi}(x)\ \varphi(x))dx\\
&=\varepsilon \int_B\sigma_B (x, |\nabla u^{f+\varepsilon\varphi}(x) |)\nabla u^{f+\varepsilon\varphi}(x)\cdot \nabla  \varphi(x)dx,
\end{split}
\]
where in the first equality we exploited that $\dive(\sigma_B(x, |\nabla u^{f+\varepsilon\varphi}|)\nabla u^{f+\varepsilon\varphi})=0$, since $u^{f+\varepsilon\varphi}$ is a solution of \eqref{problem_PEC}; in the second equality we applied the divergence Theorem, since $c^i_{\varepsilon}$ (resp. $c^i_0$) the constant value of $u^{f+\varepsilon\varphi}$ (resp. $u^{f}$) on $\partial A_i$, $i=1,..,M$; in the third equality we applied the divergence Theorem, together with the third constraint of \eqref{problem_PEC} but written for $u^{f+\varepsilon\varphi}$ and, in the fourth 
equality, we exploited the fact that $\dive(\sigma_B(x, |\nabla u^{f+\varepsilon\varphi}|)\nabla u^{f+\varepsilon\varphi}(x)=0$. Similar considerations can be applied to term $II$.

Conclusion  \eqref{fund_conv_PEC} follows by arguing as for Claim (i), but using  \eqref{lemma_traccia_PEC} from Lemma \ref{X_trace_inequality_lem}.
\end{proof}

\begin{rem}\label{rem_conv_J}
The continuity of $\sigma$ with respect to $E$ (assumption (P1)), together with the assumptions of Proposition \ref{lpconv2}, implies the following convergence result for the electrical current density $\bf J$:
    \[
    {\bf J}(x,|\nabla u^{f+\varepsilon \varphi}(x)|)\to {\bf J}(x,|\nabla u^f(x)|)\qquad\text{as}\ {\varepsilon\to 0^+},\qquad\forall x\in \Omega.
    \]
This convergence result holds in the same Sobolev spaces required in each of cases (i), (ii), (iii) and (iv) of Proposition \ref{lpconv2}.
\end{rem}

\section{The connection between the Dirichlet Energy and the DtN} \label{connection_sec}
The main aim of this Section is to prove that, in each of the four cases of nonlinearity we are interested in, there is a connection between the Dirichlet Energy and the DtN operator.

The DtN operator is of paramount relevance in inverse problems, because it can be measured from boundary data, unlike the Dirichlet Energy which is an integral quantity that requires the knowledge of the solution within the domain $\Omega$. On the other hand, the Dirichlet Energy is a relevant quantity because it satisfies a Monotonicity Principle, as shown in Section \ref{monoten}.

The intimate connection between the Dirichlet Energy and the DtN operator is given by the the G\^ateaux derivative. Specifically, it turns out that the G\^ateaux derivative of the mapping $f \mapsto \mathbb{E}_\sigma \circ \mathbb{U}_\sigma$ is equal to $\Lambda_\sigma$, where $\mathbb{U}_\sigma$ is the operator mapping the boundary data $f$ into the corresponding solution $u^f$:
\begin{equation*}
%\label{U}
\mathbb{U}_\sigma:f\in X^p_\diamond(\partial\Omega)\mapsto u^f\in W,%\in W^{1,r}(\Omega)
\end{equation*}%and $r=\min\{p,q\}$.
where $W$ is a proper functional space (see Proposition \ref{gateauxprop}). In line with \cite{corboesposito2021monotonicity}, it is possible to prove the following Proposition.
\begin{prop} \label{gateauxprop} Let $1<p,q<+\infty$, $f\in X_\diamond^p(\partial\Omega)$ and ${\sigma}$ satisfying (P1), (P2), (P3). Then 
\[
\text{\emph d} (\mathbb E_\sigma\circ\mathbb{U_\sigma})(f)=\Lambda_\sigma(f).
\]
\end{prop}
\begin{proof} In the following we prove that
\begin{equation}
\label{passaggio_tesi}
\text{\emph d}(\mathbb E_\sigma\circ\mathbb{U_\sigma}) (f;\varphi)=\langle\Lambda_\sigma (f),\varphi\rangle\quad \forall\varphi\in W 
\end{equation}
where
\begin{itemize}
\item[(i)] if $A=\emptyset$, then $W=W^{1,p}(\Omega)$;
\item[(ii)] if ${\sigma}$ satisfies (Q1)-(Q2), then $ W=W^{1,r}(\Omega)$ where $r=\min\{p,q\}$;
\item[(iii)] if ${\sigma}$ satisfies (R), then $W= W^{1,p}(B)$;
\item[(iv)] if ${\sigma}$ satisfies (S), then $W= W^{1,p}(B)$.
\end{itemize}

Since the $\sigma(x,E)E$ product is strictly increasing w.r.t. $E>0$ and $\partial_E Q_\sigma(x,E)=\sigma(x,E)E$, we have 
\begin{equation}
\label{monotonia_Q}
    \begin{split}
0\leq \sigma (x, E_1)E_1%\partial_E Q_\sigma (x, E_1)
(E_2-E_1)& \leq Q_\sigma(x, E_2)-Q_\sigma(x, E_1)\leq \sigma(x,E_2)E_2%\partial_E Q_\sigma (x, E_2)
(E_2-E_1),
    \end{split}
\end{equation}
for a.e. $x\in\overline\Omega$ and for any $ 0<E_1\leq E_2$.

$\bullet$ {\bf Case (i)}. Let $u^{f+\varepsilon \varphi}\in W^{1,p}(\Omega)$ (resp. $u^{f}\in W^{1,p}(\Omega)$) be the minimizer of \eqref{minimum_BA} corresponding to $f+\varepsilon\varphi$ (resp. $f$).
For the sake of simplicity, we set
\begin{equation}
\label{funzione_G}
\mathbb G_\sigma(f):=(\mathbb E_\sigma\circ\mathbb{U}_\sigma)(f)=\mathbb{E}_\sigma\left(  u^f \right).
\end{equation}
For any $\varepsilon>0$, we consider the following finite quantities
\begin{align}
\label{magg_G_f_Omega}
\mathbb G_\sigma(f)&=\int_\Omega Q_\sigma(x, |\nabla u^f(x)|)\text{d}x\leq \int_\Omega Q_\sigma\left(x, \left|\nabla u^{f+\varepsilon \varphi}(x)-\varepsilon\nabla \varphi(x)\right|\right)\text{d}x,\\
\label{magg_G_fep_Omega}
\mathbb G_\sigma(f+\varepsilon\varphi)&= \int_\Omega Q_\sigma\left(x, \left|\nabla u^{f+\varepsilon\varphi}(x)\right|\right)\text{d}x\leq \int_\Omega Q_\sigma\left(x, \left|\nabla u^{f}(x)+\varepsilon\nabla \varphi(x)\right|\right)\text{d}x, 
\end{align}
where we have used the minimality of the Dirichlet Energy on the solutions.

Let us consider the incremental ratio $\left[\mathbb G_\sigma(f+\varepsilon\varphi)-\mathbb G_\sigma (f)\right]/\varepsilon$. From \eqref{magg_G_fep_Omega}, we have
\begin{equation}
\label{rapp_incrG_Omega}
\frac{\mathbb G_\sigma(f+\varepsilon\varphi)-\mathbb G_\sigma (f)}\varepsilon\leq \frac 1{ \varepsilon }\int_\Omega 
Q_\sigma\left(x,|\nabla u^{f}+\varepsilon \nabla \varphi(x)|\right)-Q_\sigma\left(x, |\nabla u^f(x)|
\right)\text{d}x.
\end{equation}
The magnitude of the integrand in \eqref{rapp_incrG_Omega} can be easily bounded from above. Indeed, from \eqref{monotonia_Q}, we have
\begin{equation}
\label{magg_ri_Omega}
\begin{split}
& \left| \frac {
Q_\sigma\left(x,|\nabla {u}(x)+\varepsilon \nabla \varphi(x)|\right)-Q_\sigma\left(x, |\nabla u(x)|
\right)}\varepsilon\right|\\
&\qquad\leq\frac 1{| \varepsilon|} \sigma \left(x,\left| \nabla {u}(x)|+|\varepsilon \nabla \varphi(x)\right|\right)(\left|\nabla u(x)|+|\varepsilon \nabla \varphi(x)\right|)|\varepsilon\nabla \varphi(x)| \\
%&\qquad= \sigma \left(x,\left| \nabla {u}(x)|+|\varepsilon \nabla \varphi(x)\right|\right)(\left|\nabla u(x)|+|\varepsilon \nabla \varphi(x)\right|)|\nabla \varphi(x)| \\
&\qquad\leq \sigma \left(x,\left| \nabla u(x)|+| \nabla \varphi(x)\right|\right)(\left|\nabla u(x)|+| \nabla \varphi(x)\right|)|\nabla \varphi(x)|, 
\end{split}
\end{equation} 
for any $\varepsilon<1$. From the assumption (P2), the last term in \eqref{magg_ri_Omega} is a $L^1$ function and hence, through the Lebesgue Dominate Convergence Theorem, we can pass to the limit in \eqref{rapp_incrG_Omega}:
\begin{equation}
\label{limsup_Omega}
\begin{split}
&\limsup_{\varepsilon\to 0^+} \frac{\mathbb G_\sigma(f+\varepsilon\varphi)-\mathbb G_\sigma (f)}\varepsilon\\
&\leq\lim_{\varepsilon\to 0^+} \frac 1{ \varepsilon }\int_\Omega Q_\sigma\left(x,|\nabla u^{f}(x)+\varepsilon \nabla \varphi(x)|\right)-Q_\sigma\left(x, |\nabla u^f(x)|\right)\text{d}x\\
&=\int_\Omega  {\sigma}\left( x, \left\vert \nabla u^f(x)\right\vert\right)  \nabla u^{f}(x)\cdot \nabla \varphi (x)\ \text{\emph d}x.
\end{split}
\end{equation} 

To prove that the limit of the incremental ratio on the l.h.s. of \eqref{rapp_incrG_Omega} exists, for $\varepsilon\to 0^+$, it is sufficient to prove that the $\liminf$ is greater than or equal to the r.h.s. of \eqref{limsup_Omega}, that is
\begin{equation}
    \label{liminf_rapp_incr}
\liminf_{\varepsilon\to 0^+} \frac{\mathbb G_\sigma(f+\varepsilon\varphi)-\mathbb G_\sigma(f)}\varepsilon\ge\int_\Omega  {\sigma}\left( x, \left\vert \nabla u^f(x)\right\vert\right)  \nabla u^{f}(x)\cdot \nabla \varphi (x)\ \text{\emph d}x.
\end{equation}
Specifically, let $\{\varepsilon_j\}_{j\in\N}$ be a sequence such that $\varepsilon_j\to 0^+$ and
\begin{equation*}%    \label{limj_Omega}
\liminf_{\varepsilon\to 0^+} \frac{\mathbb G_\sigma(f+\varepsilon\varphi)-\mathbb G_\sigma(f)}\varepsilon=\lim_{j\to +\infty}  \frac{\mathbb G_\sigma(f+\varepsilon_j\varphi)-\mathbb G_\sigma(f)}{\varepsilon_j}.
\end{equation*}
By using \eqref{magg_G_f_Omega}, we have
\begin{equation}\label{lowlim_Omega}
\begin{split}
\frac{\mathbb G_\sigma(f+\varepsilon_j\varphi)-\mathbb G_\sigma(f)}{\varepsilon_j} & \geq  \frac 1{ \varepsilon_j }\int_\Omega Q_\sigma\left(x,|\nabla u^{f+\varepsilon_j \varphi}(x)|\right)-Q_\sigma\left(x, |\nabla u^{f+\varepsilon_j\varphi}(x)-\varepsilon_j \nabla\varphi(x)|\right)\text{d}x.
\end{split}
\end{equation}

To show that it is possible to pass to the limit, as $j\to + \infty$, in the r.h.s of \eqref{lowlim_Omega}, we prove that (i) the integrand is convergent and (ii)  there exists a dominating summable function for the integrand. Then, the limit exists thanks to the Dominated Convergence Theorem. 

In order to prove the pointwise convergence of the argument of the r.h.s. of \eqref{lowlim_Omega}, we observe that $Q_\sigma(x, \cdot)$ is in $C^1([0,+\infty[)$ for a.e. $x\in\Omega$ and, therefore, the integrand of the r.h.s. of \eqref{lowlim_Omega} can be written as 
\begin{equation}
\label{rapp_incr_intG}
\begin{split}
&\frac{Q_\sigma\left(x,|\nabla u^{f+\varepsilon\varphi}(x)|\right)-Q_\sigma\left(x, |\nabla u^{f+\varepsilon \varphi}(x)-\varepsilon\nabla \varphi(x)|\right)}\varepsilon\\
&\qquad\qquad=\sigma \left(x,\xi_\varepsilon' \right)\xi_\varepsilon' \frac{ \left|\nabla u^{f+\varepsilon\varphi}(x)\right|-\left|\nabla u^{f+\varepsilon\varphi}(x)-\varepsilon\nabla \varphi(x)\right|}\varepsilon\\
\end{split}
\end{equation}
with
\begin{equation*}%\label{intervalG}
\begin{split}
\xi_\varepsilon'\in &[\min\{\left|\nabla u^{f+\varepsilon \varphi}(x)\right|, \left|\nabla u^{f+\varepsilon \varphi}(x)-\varepsilon \nabla \varphi(x)\right|\},\\ 
&\qquad\qquad\qquad\qquad\qquad\qquad\max\{\left|\nabla u^{f+\varepsilon \varphi}(x)\right|,\left|\nabla u^{f+\varepsilon\varphi}(x)-\varepsilon \nabla \varphi(x)\right|\}].
\end{split}
\end{equation*}
%where the dependence of $\xi_\varepsilon'$ upon $x$ is apparent.
The r.h.s. in \eqref{rapp_incr_intG} is equal to
\begin{equation}\label{last_Omega}
\begin{sistema}
\sigma \left(x,\xi_\varepsilon' \right)\xi_\varepsilon' \dfrac{2\nabla u^{f+\varepsilon\varphi}(x)\cdot\nabla \varphi (x)-\varepsilon \left|\nabla\varphi(x)\right|^2}{|\nabla u^{f+\varepsilon\varphi}(x)|+|\nabla u^{f+\varepsilon\varphi}(x)-\varepsilon \nabla \varphi(x)|}\quad\text{if}\ |\nabla u^{f+\varepsilon\varphi}(x)|\neq 0,\vspace{0.2cm}\\
\sigma \left(x,\xi_\varepsilon' \right)\xi_\varepsilon' \sign(-\varepsilon)|\nabla\varphi (x)|\qquad\qquad\qquad\qquad\qquad \text{otherwise}.
\end{sistema}
\end{equation}
Exploiting the fact that $\xi_\varepsilon'\to |\nabla u^f(x)|$ as $\varepsilon\to 0^+$ for a.e. $x\in\Omega$ thanks to \eqref{fund_conv_Omega}, we find that both expression appearing in \eqref{last_Omega} tend to ${\sigma}\left( x, \left\vert \nabla u^f(x)\right\vert\right)  \nabla u^f(x)\cdot \nabla \varphi (x)$, as $\varepsilon\to 0^+$. 

To prove the existence of a summable dominating function for the integrand of \eqref{lowlim_Omega}, we recall from Proposition \ref{lpconv2} (see the convergence result \eqref{fund_conv_Omega}), that $\nabla u^{f+\varepsilon\varphi}\to\nabla u^f$ strongly in $L^p(\Omega)$, when $\varepsilon\to 0^+$.
By standard arguments, we can say that there exists a subsequence $\{\varepsilon_{j_h}\}_{h\in\N}$ such that $\nabla u^{f+\varepsilon_{j_h}\varphi}\to\nabla u^f$ a.e. in $\Omega$, and there exists a measurable real function $\psi$, defined on $\Omega$, such that
\begin{equation}
\label{dominating_function_Omega}
|\nabla u^{f+\varepsilon_{j_h}\varphi}|\leq\psi\ \text{in}\ \Omega,\quad\int_\Omega\psi(x)^p dx<+\infty.
\end{equation}
When $p\ge 2$, for any $\varepsilon_{j_h}<1$, we have
\begin{equation}
\label{dominazione_Omega}
\begin{split}
&\left|\frac{Q_\sigma\left(x, |\nabla u^{f+\varepsilon_{j_h}\varphi}(x)-\varepsilon_{j_h} \nabla\varphi(x)|\right)-Q_\sigma\left(x,|\nabla u^{f+\varepsilon_{j_h} \varphi}(x)|\right)}{\varepsilon_{j_h}} \right|\\
&\le \sigma(x,|\nabla u^{f+\varepsilon_{j_h}\varphi}(x)|+|\varepsilon_{j_h}  \nabla\varphi(x)| )( |\nabla u^{f+\varepsilon_{j_h} \varphi}(x)|+|\varepsilon_{j_h} \nabla\varphi(x)|)|\nabla \varphi(x)|\\
&\leq \overline\sigma  \left[1+\left(\frac{|\nabla u^{f+\varepsilon_{j_h}\varphi}(x)|+|\varepsilon_{j_h}  \nabla\varphi(x)|}{E_0}\right )^{p-2}\right] \left( |\nabla u^{f+\varepsilon_{j_h}\varphi}(x)|+|\varepsilon_{j_h} \nabla\varphi(x)|\right)|\nabla\varphi(x)|\\
&\leq \overline\sigma  \left[1+\left(\frac{\psi(x)+ | \nabla\varphi(x)| }{E_0}\right)^{p-2}\right]( \psi(x) + |\nabla\varphi(x)|)^2\\
&\leq \overline\sigma \left[ \left ( {\psi(x) + |\nabla\varphi(x)|}\right)^2+ \left(\frac{\psi(x) + | \nabla\varphi(x)|}{E_0}\right)^{p} \right],
\end{split}
\end{equation}
where in the first inequality we have used \eqref{monotonia_Q} and the convexity of $Q_\sigma(\cdot,E)$ (see Remark \ref{rem_ipotesi_lin}); in the second inequality we have used (P2); in the third inequality we have used \eqref{dominating_function_Omega}; the fourth inequality is straightforward. It is worth noting that the inequalities chain \eqref{dominazione_Omega} for $1<p<2$ can be similarly proven by using the second line of the assumption (P2).

Since $\overline\sigma$ in \eqref{dominazione_Omega} is independent of $\varepsilon_{j_h}$, the Dominate Convergence Theorem provides \eqref{liminf_rapp_incr} for any $p>1$.

Summing up, from \eqref{limsup_Omega} and \eqref{liminf_rapp_incr}, it turns out that 
\begin{equation*}%\label{liminf_conv}
\lim_{\varepsilon \to 0^+}  \frac{\mathbb G_\sigma(f+\varepsilon\varphi)-\mathbb G_\sigma(f)}{\varepsilon}\\
= \int_\Omega  {\sigma}\left( x, |\nabla u^f(x)|\right) \nabla u^{f}(x)\cdot\nabla \varphi (x) \text{d}x.
\end{equation*}
The conclusion follows by observing that this equation is nothing but \eqref{passaggio_tesi}.

$\bullet$ {\bf Case (ii)}. 
Let $u^{f+\varepsilon \varphi}\in W^{1,r}(\Omega)$ (resp. $u^{f}\in W^{1,r}(\Omega)$) be the minimizer of \eqref{minimum_BA} corresponding to $f+\varepsilon\varphi$ (resp. $f$). As in case (i), we consider the function $\mathbb G_\sigma$ defined in \eqref{funzione_G}. Analogously to case (i), we can upper bound the lim sup and lower bound the lim inf of the incremental ratio $\left[\mathbb G_\sigma(f+\varepsilon\varphi)-\mathbb G_\sigma (f)\right]/\varepsilon$ with the same quantity.

Regarding the $\limsup$, arguing as in case (i), we obtain \eqref{limsup_Omega}, that is
\begin{equation}
\label{limsup_AB}
\begin{split}
\limsup_{\varepsilon\to 0^+} \frac{\mathbb G_\sigma(f+\varepsilon\varphi)-\mathbb G_\sigma (f)}\varepsilon\le \int_\Omega  {\sigma}\left( x, \left\vert \nabla u^f(x)\right\vert\right)  \nabla u^{f}(x)\cdot \nabla \varphi (x)\ \text{\emph d}x.
\end{split}
\end{equation} 

Regarding the $\liminf$, we have to show that it is possible to pass to the limit on the r.h.s. of \eqref{lowlim_Omega} by using the Dominated Convergence Theorem. Following line by line the analogous part of the proof of case  (i), we can show that the integrand is convergent to ${\sigma}\left( x, \left\vert \nabla u^f(x)\right\vert\right)  \nabla u^f(x)\cdot \nabla \varphi (x)$, as $\varepsilon\to 0^+$. Therefore, it remains to prove that there exists a dominating summable function for the integrand appearing on the r.h.s of \eqref{lowlim_Omega}. 
To prove this, we notice that in Proposition \ref{lpconv2} (see the convergence result \eqref{fund_conv_AB}), we have shown that $\nabla u^{f+\varepsilon\varphi}\to\nabla u^f$ in $L^p(B)$ and $\nabla u^{f+\varepsilon\varphi}\to\nabla u^f$ in $L^q(A)$, when $\varepsilon\to 0^+$. Then by standard arguments, we can say that there exists a subsequence $\{\varepsilon_{j_h}\}_{h\in\N}$ such that  $\nabla u^{f+\varepsilon_{j_h}\varphi}\to\nabla u^f$ a.e. in $\Omega$, and there exists a measurable real function $\psi$, defined on $\Omega$, such that
\begin{equation}
\label{dominating_function_AB}
|\nabla u^{f+\varepsilon_{j_h}\varphi}|\leq\psi\ \text{in}\ \Omega%B,\quad |\nabla u^{f+\varepsilon_{j_h}\varphi}|\leq\psi\ \text{in}\ A
,\quad\int_B\psi(x)^p dx+\int_A\psi(x)^q dx<+\infty.
\end{equation}
For any $p,q\ge 2$ and for any $\varepsilon_{j_h}$, we have
\begin{equation}
\label{dominazione_AB}
\begin{split}
&\int_\Omega\left|\frac{Q_\sigma\left(x, |\nabla u^{f+\varepsilon_{j_h}\varphi}(x)-\varepsilon_{j_h} \nabla\varphi(x)|\right)-Q_\sigma\left(x,|\nabla u^{f+\varepsilon_{j_h} \varphi}(x)|\right)}{\varepsilon_{j_h}} \right|dx\\
&\leq\displaystyle\int_\Omega \sigma(x,|\nabla u^{f+\varepsilon_{j_h}\varphi}(x)|+|\varepsilon_{j_h}  \nabla\varphi(x)| )( |\nabla u^{f+\varepsilon_{j_h} \varphi}(x)|+|\varepsilon_{j_h} \nabla\varphi(x)|)|\nabla \varphi(x)|dx\\
&\leq \overline{\sigma} \displaystyle\int_B \left[1+\left(\frac{|\nabla u^{f+\varepsilon_{j_h}\varphi}(x)|+|\varepsilon_{j_h}  \nabla\varphi(x)|}{E_0}\right )^{p-2}\right] \left( |\nabla u^{f+\varepsilon_{j_h}\varphi}(x)|+|\varepsilon_{j_h} \nabla\varphi(x)|\right)|\nabla\varphi(x)|dx\\
&\ \ + \overline{\sigma} \displaystyle\int_A \left[1+\left(\frac{|\nabla u^{f+\varepsilon_{j_h}\varphi}(x)|+|\varepsilon_{j_h}  \nabla\varphi(x)| }{E_0}\right)^{q-2}\right] \left( |\nabla u^{f+\varepsilon_{j_h}\varphi}(x)|+|\varepsilon_{j_h} \nabla\varphi(x)|\right)|\nabla\varphi(x)|dx\\
&\leq \overline{\sigma} \displaystyle\int_B \left[1+\left(\frac{\psi(x)+ | \nabla\varphi(x)| }{E_0}\right)^{p-2}\right]( \psi(x)+ |\nabla\varphi(x)|)^2dx\\
&\ \ + \overline{\sigma} \displaystyle\int_A \left[1+\left(\frac{\psi(x)+ | \nabla\varphi(x)| }{E_0}\right)^{q-2}\right]( \psi(x)+ |\nabla\varphi(x)|)^2dx\\
&\leq \overline{\sigma}\displaystyle\int_B \left (\psi(x)+ |\nabla\varphi(x)|\right)^2+ \left(\frac{\psi(x)+ | \nabla\varphi(x)|}{E_0}\right)^{p}dx\\
&\ \ + \overline{\sigma}\displaystyle\int_A \left ( \psi(x)+ |\nabla\varphi(x)|\right)^2+ \left(\frac{\psi(x)+ | \nabla\varphi(x)|}{E_0}\right)^{q}dx,
\end{split}
\end{equation}
where in the first inequality we have used \eqref{monotonia_Q} and the convexity of $Q_\sigma(\cdot,E)$ (see Remark \ref{rem_ipotesi_lin}); in the second inequality we have used (P2); in the third inequality we have used \eqref{dominating_function_AB}; the fourth inequality is straightforward.

The inequalities chain \eqref{dominazione_AB} for $1<p<2$ can be  similarly proven by using the second line of the assumption (P2). 

By applying the Dominate Convergence Theorem, the limit as $\varepsilon_{j_h}\to 0$ in \eqref{lowlim_Omega} provides
\begin{equation}
\label{liminf_AB}
\begin{split}
\liminf_{\varepsilon\to 0^+} \frac{\mathbb G_\sigma(f+\varepsilon\varphi)-\mathbb G_\sigma (f)}\varepsilon\ge \int_\Omega  {\sigma}\left( x, \left\vert \nabla u^f(x)\right\vert\right)  \nabla u^{f}(x)\cdot \nabla \varphi (x)\ \text{\emph d}x.
\end{split}
\end{equation} 

Finally, the conclusion follows by combining \eqref{limsup_AB} and \eqref{liminf_AB}.

$\bullet$ {\bf Case (iii)}. The proof is obtained as in case (ii) but replacing the role of $\Omega$ with $B$ and exploiting the strong convergence result of Proposition \ref{lpconv2}, case (iii).

$\bullet$ {\bf Claim (iv)}. The proof is obtained as in case (ii) but replacing the role of $\Omega$ with $B$ and exploiting the strong convergence result of Proposition \ref{lpconv2}, case (iv).
\end{proof}
As a straightforward consequence of Proposition \ref{gateauxprop}, we have the following expressions for the derivative of $\mathbb E_\sigma$, with respect to $u$ and $u^f$, respectively.

\begin{prop} Let $f\in X_\diamond(\partial \Omega)$, ${\sigma}$ satisfying (P1), (P2) and (P3), and $u,\varphi\in W$, then
\begin{align*}
&\text{\emph d}\mathbb F_\sigma (u;\varphi)=\int_\Omega  {\sigma}\left( x, \left\vert \nabla u(x)\right\vert\right)  \nabla u(x)\cdot \nabla \varphi (x)\ \text{\emph d}x,\\
&\text{\emph d}\mathbb F_\sigma (u^f;\varphi)=\langle\Lambda_\sigma (f),\varphi\rangle,
\end{align*} 
where
\begin{itemize}
\item[(i)] if $A=\emptyset$, then $W=W^{1,p}(\Omega)$ and $u^{f}$ is the minimizer of \eqref{minimum_BA};
\item[(ii)] if ${\sigma}$ satisfies (Q1)-(Q2), then $ W=W^{1,r}(\Omega)$ where $r=\min\{p,q\}$ and $u^{f}$ is the minimizer of \eqref{minimum_BA};
\item[(iii)] if ${\sigma}$ satisfies (R), then $W= W^{1,p}(B)$ and $u^{f}$ is the minimizer of \eqref{minimum_PEI};
\item[(iv)] if ${\sigma}$ satisfies (S), then $W= W^{1,p}(B)$ and $u^{f}$ is the minimizer of \eqref{minimum_PEC}.
\end{itemize}
\end{prop}

 \section{The Monotonicity Principle}
 \label{monotonicity_sec}
Here we prove the Monotonicity Principle of the average-DtN, with respect to the electrical conductivity.
%The inverse problem consists in retrieving the electrical conductivity $\sigma$ starting from boundary measurements, as those given by $\Lambda_\sigma(f)$. In this framework, Monotonicity Principles have a key role. They consist in monotone relations among electrical conductivities and boundary operators, as in the $p$-Laplacian \eqref{pOhm_potential} and in the linear \eqref{Ohm_potential} cases:\begin{equation*}\sigma_1\leq\sigma_2\ \text{for a.e.} \ x\in\Omega\quad\Longrightarrow\quad\langle \Lambda _{\sigma_1 }\left( f\right),f\rangle \leq \langle\Lambda _{\sigma_2 }\left( f\right),f\rangle,\quad\forall f\in X_\diamond,\end{equation*}with the objective to retrieve, in a non destructive way, the electrical conductivity starting from the knowledge of boundary data. 
%In the appendix of \cite{gisser1990electric} there is one of the first evidence of Monotonicity properties for the linear case \eqref{Ohm_potential}. Subsequently, the authors in \cite{Tamburrino_2002,Tamburrino2006FastMF} firstly used Monotonicity in the framework of inverse problems to develop reconstruction imaging methods and algorithms, useful in electric resistance tomography field. In particular, they give the proof in the discrete case, when a finite number of electrodes are applied. %In particular, they showed a monotone relation for the DtN operators starting from a monotone relation the electrical  electrical conductivities and, subsequently, in \cite{H} also the inverse has been showed. Furthermore, in \cite{doi:10.1137/110838224}, a Monotonicity Principle has been showed for $p$-Laplacian case.

Specifically, first we demonstrate a Monotonicity Principle for the Dirichlet Energy (Theorem \ref{monoten})
\begin{equation*}%\label{m_genergy3}
\sigma_1\leq\sigma_2\quad\Longrightarrow\quad\mathbb E_{\sigma_1}(u_{1}^f) \leq \mathbb E_{\sigma_2} (u_{2}^f)\quad \forall f\in X_\diamond^p(\partial\Omega),
\end{equation*}
where $u_{i}^f$ is the minimizer of the Dirichlet Energy with $\sigma=\sigma_i$, for $i=1,2$ and $f$ is the applied boundary voltage and we subsequently \lq\lq translate\rq\rq \ the Dirichlet Energy in terms of boundary data via the following fundamental relationship (see Theorem \ref{transferthm})
\begin{equation}
\label{translation}
\mathbb{E}_{\sigma}\left( u^f\right)=\left\langle\overline{\Lambda}_\sigma  \left( f\right) ,f \right\rangle%\left\langle  \int_{0}^{1} \Lambda_\sigma  \left( \alpha f\right) \text{d} \alpha f \right\rangle,
\quad\forall f\in X^p_\diamond(\partial\Omega).
\end{equation}
The operator $\overline{\Lambda}_\sigma$ appearing in \eqref{translation} is called Average DtN and is defined as
\begin{equation*}%\label{average_flown}
% \int_{0}^{1}\Lambda  \left( \alpha f\right) \text{d}\alpha 
\overline{\Lambda}_\sigma: f\in X^p_\diamond(\partial\Omega)\mapsto  \int_{0}^{1}\Lambda_\sigma  \left( \alpha f\right) \text{d}\alpha\in X^p_\diamond(\partial\Omega)',
\end{equation*}
where
\begin{equation*}
\langle\overline{\Lambda}_\sigma( f),\varphi\rangle=
 \int_{0}^{1}\left\langle\Lambda_\sigma  \left( \alpha f\right) , \varphi\right\rangle\text{d}\alpha\quad\forall\varphi \in X^p_\diamond(\partial\Omega).
\end{equation*} 
Operator $\overline{\Lambda}_\sigma$ gives the average flow of the electrical current density through $\partial \Omega$ for an applied boundary potential of the type $\alpha f$, for $\alpha \in \left[0, 1\right]$. It is the key \lq\lq tool\rq\rq \ for transferring the Monotonicity Principle for the Dirichlet Energy to the boundary data and it replaces the usual DtN operator $\Lambda_\sigma$ when treating nonlinear problems (see \cite{corboesposito2021monotonicity} for its very first introduction).

Finally, as a byproduct of Theorems \ref{monoten} and \ref{transferthm}, we get the Monotonicity Principle for the operator $\overline{\Lambda}_\sigma$ (see Theorem \ref{monothm}), i.e.
\begin{equation}
\label{MonoP}    
\sigma_1\leq\sigma_2\quad\Longrightarrow\quad \left\langle\overline{\Lambda}_{\sigma_1}  \left( f\right) ,f \right\rangle
\leq \left\langle\overline{\Lambda}_{\sigma_2}  \left( f\right) ,f \right\rangle \quad\forall f\in X_\diamond^p(\partial\Omega).
\end{equation}
In \eqref{MonoP} $\sigma_1\leq\sigma_2$ means that
\begin{equation}
    \label{defmon}
\sigma_1(x,E)\leq \sigma_2(x,E) \quad \text{for a.e.}\ x\in\overline\Omega\ \text{and}\ \forall\ E>0.
\end{equation}
Furthermore, let us also observe that there exist four possibly different subsets $B_i, A_i$ of $\Omega$, $i=1,2$, such that $\sigma_i$ is defined on $\Omega=B_i\cup A_i$, $i=1,2$, with the same growth exponent $p$ on $B_i$ and possibly different growth exponents $q_i$ on $A_i$, $i=1,2$. Therefore, no inclusion relation is imposed between $A_1$ and $A_2$ but only the validity of relation \eqref{defmon}.

\begin{rem}
It is worth noting that the operator
\begin{equation}
\Pi:\sigma\in L^\infty_+(\Omega) \mapsto \overline{\Lambda}_\sigma
\end{equation}
%\Pi:\sigma\in L^\infty_+(\Omega) \mapsto \left( f \in  X_\diamond^p(\partial\Omega) \mapsto \overline{\Lambda}_{\sigma}(f) \in X_\diamond^p(\partial\Omega)' \right)
is monotonic in the sense of the MP stated in \eqref{MonoP}, where $\overline{\Lambda}_\sigma$ the average DtN operator defined as in \eqref{average_DtN}. On the hand, for a prescribed $\sigma$, the average DtN is a (non linear) monotonic operator in the classical sense stated in the books of Brezis \cite{brezis1973ope} and Zeidler \cite{zeidler1987nonlinear}, i.e.
\begin{equation}
   \left\langle\overline{\Lambda}_{\sigma}  \left( f_1\right)-\overline{\Lambda}_{\sigma}  \left( f_2\right) ,f_1-f_2 \right\rangle \ge 0 \quad\forall f_1, f_2\in X_\diamond^p(\partial\Omega).
\end{equation}
\end{rem}

The Monotonicity Principle for the Dirichlet Energy in the nonlinear case can be easily proved as follows.
\begin{thm}\label{monoten}
Let $1<p,q<+\infty$, $\sigma_1$, $\sigma_2$ satisfying (P1), (P2) and (P3), then
\begin{equation*}
\sigma_1\leq\sigma_2\quad\Longrightarrow\quad\mathbb E_{\sigma_1}(u_{1}^f) \leq \mathbb E_{\sigma_2} (u_{2}^f)\quad \forall f\in X_\diamond^p(\partial\Omega),
\end{equation*}
where $\sigma_1\leq \sigma_2$ is meant in the sense \eqref{defmon} and
\begin{itemize}
\item[(i)] if $A=\emptyset$, then $u_1^{f}$ (resp. $u_2^{f}$) is the minimizer of \eqref{minimum_BA} corresponding to $\sigma_1$ (resp. $\sigma_2$);
\item[(ii)] if $\sigma_1, \sigma_2$ satisfy (Q1)-(Q2), then $u_1^{f}$ (resp. $u_2^{f}$) is the minimizer of \eqref{minimum_BA} corresponding to $\sigma_1$ (resp. $\sigma_2$);
\item[(iii)] if $\sigma_1, \sigma_2$ satisfy (R), then $u_1^{f}$ (resp. $u_2^{f}$) is the minimizer of \eqref{minimum_PEI} corresponding to $\sigma_1$ (resp. $\sigma_2$)
\item[(iv)] if $\sigma_1, \sigma_2$ satisfy (S), then then $u_1^{f}$ (resp. $u_2^{f}$) is the minimizer of \eqref{minimum_PEC} corresponding to $\sigma_1$ (resp. $\sigma_2$).
\end{itemize}
\end{thm}
\begin{proof}
Since $u_{2}^f$ is an admissible function for problem \eqref{energy_BA} with $\sigma=\sigma_1$, we have
\begin{equation*}
\begin{split}
\mathbb{E}_{\sigma_1}( u_{1}^f)& \leq \mathbb{E}_{\sigma_1}(u_{2}^f)  =\int_\Omega Q_{\sigma_1} (x,|\nabla u_{2}^f(x)|)\ \text{d}x\\
& = \int_\Omega \int_{0}^{|\nabla u_{2}^f(x)|}\sigma_1\left( x,\xi\right) \xi\ \text{d}\xi\ \text{d}x \leq \int_\Omega \int_{0}^{|\nabla u_{2}^f(x)|} \sigma_2\left( x,\xi\right)\xi\  \text{d}\xi\ \text{d}x= \mathbb E_{\sigma_2}(u_{2}^f),
\end{split}
\end{equation*}
where the second inequality follows from the assumption $\sigma_1\leq\sigma_2$.
\end{proof}

The Dirichlet Energy can be transferred to a boundary measurement involving $\overline{\Lambda}_\sigma$ as follows.
\begin{thm}\label{transferthm}
Let $1<p,q<+\infty$ and $\sigma$ satisfying (P1), (P2) and (P3), then
\begin{equation*}
\mathbb{E}_{\sigma}\left( u^f\right)=\left\langle\overline{\Lambda}_\sigma \left( f\right) ,f \right\rangle 
\quad \forall f\in X_\diamond^p(\partial\Omega),
\end{equation*}
where 
\begin{itemize}
\item[(i)] if $A=\emptyset$, then $u^{f}$ is the minimizer of \eqref{minimum_BA};
\item[(ii)] if ${\sigma}$ satisfies (Q1)-(Q2), then $u^{f}$ is the minimizer of \eqref{minimum_BA};
\item[(iii)] if ${\sigma}$ satisfies (R), then then $u^{f}$ is the minimizer of \eqref{minimum_PEI};
\item[(iv)] if ${\sigma}$ satisfies (S), then then $u^{f}$ is the minimizer of \eqref{minimum_PEC}.
\end{itemize}
\end{thm}

\begin{proof}%{\it (Theorem \ref{transferthm})} 
For $\alpha\in [0,1]$%and $\mathbb G$ defined as in \eqref{G}
, we set \[
g(\alpha):=(\mathbb F_\sigma\circ\mathbb{U_\sigma})  (\alpha f).
\]
From Proposition \ref{gateauxprop}, since $\mathbb F_\sigma\circ\mathbb{U_\sigma}$ is G\^{a}teaux-differentiable, $g$ is differentiable.  By replacing $f$ and $\varphi$ with $\alpha f$ and $f$, we have 
\[
g'(\alpha)=\text{d}(\mathbb F_\sigma\circ\mathbb{U_\sigma})(\alpha f;f)=\langle \Lambda_\sigma (\alpha f),f\rangle.
\]
Eventually, by integrating over the interval $\left[0, 1\right]$, we obtain
\begin{equation*}
\mathbb{F}_\sigma (u^f)=(\mathbb F_\sigma\circ\mathbb{U_\sigma})  (f)=g(1)=\int_0^1 g'(\alpha)\text{d}\alpha=\int_0^1\langle \Lambda_\sigma (\alpha f),f\rangle\  \text{d}\alpha=  \left\langle\overline{\Lambda}_\sigma \left( f\right) ,f \right\rangle.
\end{equation*}
\end{proof}

Finally, we have the Monotonicity Principle for the average DtN $\overline{\Lambda}_\sigma$.
\begin{thm}\label{monothm}
Let $1<p,q<+\infty$ and $\sigma_1$, $\sigma_2$ satisfying (P1), (P2) and (P3), then
\begin{equation}
\label{m_charge}
\sigma_1\leq\sigma_2\quad\Longrightarrow\quad \left\langle\overline{\Lambda}_{\sigma_1}  \left( f\right) ,f \right\rangle
\leq \left\langle\overline{\Lambda}_{\sigma_2}  \left( f\right) ,f \right\rangle%\sigma_1(x,E)\leq\sigma_2(x,E)\quad\Longrightarrow\quad\left\langle   \int_{0}^{1}\Lambda _{\sigma_1 }\left( \alpha f\right) \text{d} \alpha ,f \right\rangle \leq \left\langle  \int_{0}^{1}\Lambda _{\sigma_2 }\left( \alpha f\right) \text{d}\alpha , f \right\rangle,
\quad \forall f\in X_\diamond^p(\partial\Omega),
\end{equation}
where $\sigma_1\leq \sigma_2$ is meant in the sense \eqref{defmon}, and one of the following holds
\begin{enumerate}
    \item[(i)] $A=\emptyset$;
    \item[(ii)] ${\sigma_1}, {\sigma_2}$ satisfy (Q1)-(Q2); 
    \item[(iii)] ${\sigma_1}, {\sigma_2}$ satisfy (R); 
    \item [(iv)] ${\sigma_1}, {\sigma_2}$ satisfy (S).
\end{enumerate}
\end{thm}
\begin{proof}
We consider $u_{i}^f$ to be the unique solution with $\sigma=\sigma_i$, $i=1,2$, and boundary value $f$. From Theorems \ref{monoten} and \ref{transferthm}, we have
\begin{equation*}
\begin{split}
\left\langle\overline{\Lambda}_{\sigma_1}  \left( f\right) ,f \right\rangle= \mathbb{E}_{\sigma_1}( u^f_{1}) \leq \mathbb{E}_{\sigma_2}( u^f_{2})  = \left\langle\overline{\Lambda}_{\sigma_2}  \left( f\right) ,f \right\rangle
\end{split}
\end{equation*}
that proves \eqref{m_charge}.
\end{proof}

If the constitutive relationship $J_\sigma=J_\sigma \left(x,E\right)$ is strictly decreasing with respect to $E$, then the Monotonicity is reversed, i.e.
\begin{equation*}
\sigma_1\leq\sigma_2\quad\Longrightarrow\quad \left\langle\overline{\Lambda}_{\sigma_1}  \left( f\right) ,f \right\rangle
\geq \left\langle\overline{\Lambda}_{\sigma_2}  \left( f\right) ,f \right\rangle\quad\forall f\in X_\diamond^p(\partial\Omega).
\end{equation*}

\section{Numerical Examples}
\label{applications_sec}
In this Section we provide numerical evidence of the Monotonicity Principle ~\eqref{m_charge}, proved in Theorem~\ref{monothm}. The numerical examples are inspired by a real-world application: nonlinear conductive materials in superconductive wires. In the following, we refer to a standard Bi-2212 commercial round wire. The geometry of the cable is shown in Figure~\ref{fig:HTSC}.
\begin{figure}[!ht]
\includegraphics[width=0.8\textwidth]{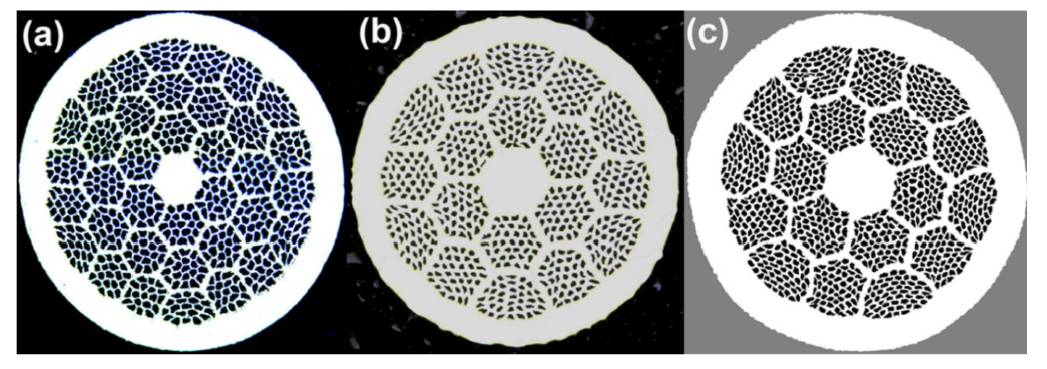}
\caption{Typical geometry for the cross section of superconducting cables. The cable is made of several petals embedded in a linear material (36 petals for (a), and 18 petals for (b) and (c)). Each petal is constituted by many thin superconductive filaments (19 wires for (a), 37 wires for (b) and 61 for (c)). This figure is reproduced with the kind permission of the publisher MDPI (see \cite[Fig. 4]{instruments4020017}).}
\label{fig:HTSC}
\end{figure}
It is made of several superconducting \lq\lq petals\rq\rq embedded in a linear matrix. The matrix is made of an AgMg alloy (Ag 0.2\%, Mg 99.8\%, electrical conductivity of $55.5\,\text{MS/m}$~\cite{li2015rrr}).

The nonlinear electrical conductivity of petals, can be derived by the so-called \emph{E-J Power Law}~\cite{rhyner1993magnetic}
\begin{equation*}
    \sigma(E) = \frac {J_c}{E_0} \left( \frac{E}{E_0} \right) ^{\frac{1-n}{n}}.
\end{equation*}
For the case of interest, we assume
\begin{displaymath}
    J_c=8000\,\text{A/mm$^2$}, \quad n=27,
\end{displaymath}
which are realistic values for these types of materials~\cite{Ba21}. The parameter $E_0=0.1\,\text{mV/m}$ is a reference value independent of the particular material~\cite{yao2019numerical}. The external radius $R_e$ is equal to 0.6 mm~\cite{huang2013bi}.

We highlight that region $A$ indicates to the petals (growth exponent $1<q<2$), while region $B$ indicates to the linear AgMg matrix (growth exponents $p=2$).

For the sake of simplicity, the numerical simulations are carried out  by considering each petal made up of a single superconductive wire, instead of many thin filaments (see Figure~\ref{fig:matdef} (left)).

As a first example, we consider a damaged wire with a crack in the AgMg matrix. Specifically, we assume the crack to be a perfect electric insulator (PEI). The geometry of the crack is shown in Figure~\ref{fig:matdef}.
\begin{figure}[htp]
\centering
\includegraphics[width=0.40\textwidth]{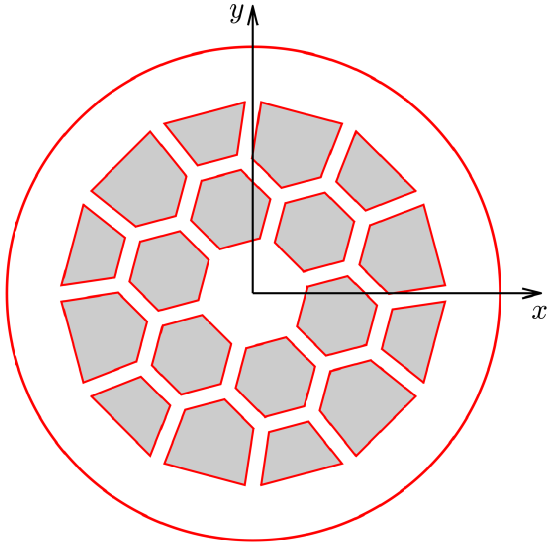}\qquad\quad \includegraphics[width=0.40\textwidth]{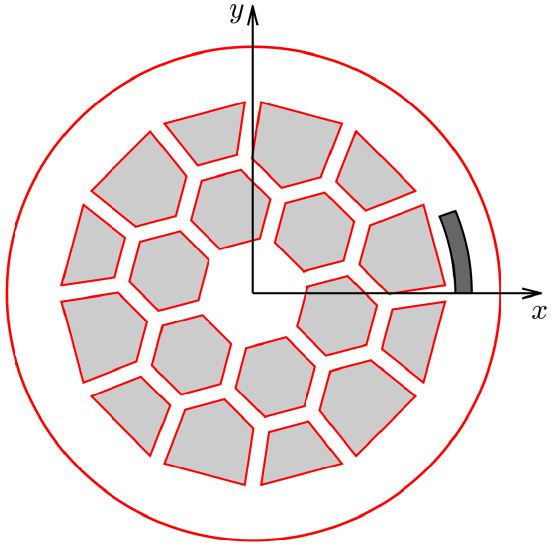}
\caption{Left: the geometry of the healthy wire. In white the linear matrix and in gray the \lq\lq petals\rq\rq \ made by the superconductive material. Right: the geometry of the damaged wire. The black region represents the crack.}
    \label{fig:matdef}
\end{figure}

Let $f$ be the applied boundary data, $\mathbb{E}_{\sigma_0}(f)=\langle\overline{\Lambda}_{\sigma_0}(f),f\rangle$ the Dirichlet Energy related to the healthy wire and $\mathbb{E}_{\sigma_1}(f)=\langle\overline{\Lambda}_{\sigma_1}(f),f\rangle$ the Dirichlet Energy related to the damaged wire. The presence of the crack reduces the electrical conductivity in region of the domain and, therefore, we expect a decrease in the Dirichlet Energy, i.e. $\mathbb{E}_{\sigma_1}(f)\geq \mathbb{E}_{\sigma_0}(f)$, $\forall f$. The numerical results in Table~\ref{tab:ex1} confirm the validity of the Monotonicity Principle.
\begin{table}[htp]
\centering
\[
\begin{array}{cccc}
\toprule
f\,\text{(V)} & \mathbb{E}_{\sigma_0} (f)\,\text{(W)}& \mathbb{E}_{\sigma_1} (f)\,\text{(W)}& \mathbb{E}_{\sigma_0}(f)\,-\mathbb{E}_{\sigma_1}(f)\,\text{(W)} \\
\midrule
100x & 6.6610\times 10^5 & 6.4155\times 10^5 & 2.4548\times 10^4 \\
300x & 3.4927\times 10^6 & 3.4194\times 10^6 & 7.3253\times 10^4 \\
500x & 7.4529\times 10^6 & 7.3145\times 10^5 & 1.3846\times 10^5 \\
0.1\sin(\theta) & 1.4723\times 10^6 & 1.4501\times 10^6 & 2.2170\times 10^4 \\
0.3\sin(\theta) & 7.4530\times 10^6 & 7.3499\times 10^6 & 1.0313\times 10^5 \\
0.5\sin(\theta) & 1.6446\times 10^7 & 1.6202\times 10^7 & 2.4423\times 10^5 \\
0.1\sin(\theta)-0.2\cos(2\theta) & 7.7056\times 10^6 & 7.5517\times 10^6 & 1.5391\times 10^5 \\
100\exp{(x^2+2y)} & 1.9288\times 10^6 & 1.9009\times 10^6 & 2.7932\times 10^4 \\
300\exp{(x^2+2y)} & 9.8394\times 10^6 & 9.7004\times 10^6 & 1.3895\times 10^5 \\
500\exp{(x^2+2y)} & 2.2051\times 10^7 & 2.1712\times 10^7 & 3.3846\times 10^5 \\
\bottomrule
\end{array}
\]
\caption{Numerical evaluation of $\mathbb{E}_{\sigma_0}(f)$ and $\mathbb{E}_{\sigma_1}(f)$ for the case of a damaged matrix. The difference $\mathbb{E}_{\sigma_0}(f)-\mathbb{E}_{\sigma_1}(f)$ is greater than $0$.}
\label{tab:ex1}
\end{table}
%It is worth noting that if we apply a scaled version of the boundary potential $\alpha f$, instead of $f$, the measured Dirichlet Energy is not equal to $\alpha^2 \mathbb{E}_0$ or $\alpha^2 \mathbb{E}_1$. This follows from the nonlinearity of the problem and it highlights the need of a non countable set of measurements to charatecterize the average DtN operator even for a fixed direction $f$. 

Another configuration of interest is that of a superconductive wire with a damaged petal, which is assumed to be a perfect insulator. The geometry is depicted in Figure~\ref{fig:ex2}.
\begin{figure}[htp]
\centering
\includegraphics[width=0.45\textwidth]{MPM_04_Sgeo2.pdf}\qquad\quad \includegraphics[width=0.45\textwidth]{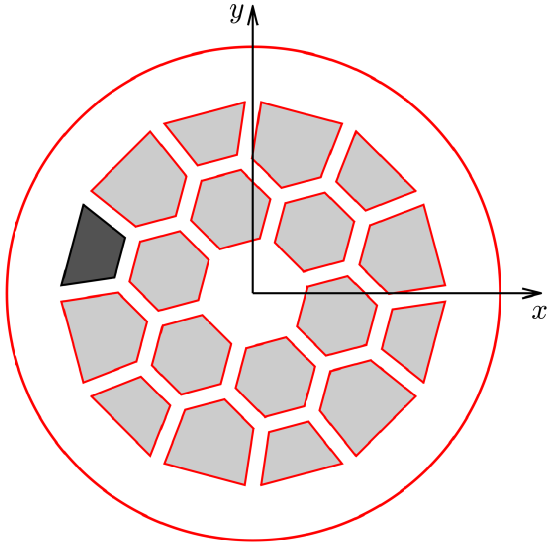}
\caption{Left: the geometry for the healthy cable. Right: the cable with a damaged petal (PEI) in black.}
\label{fig:ex2}
\end{figure}
As in the previous case, since the damage causes a loss of conductive material, we expect that the Dirichlet Energy $\mathbb{E}_{\sigma_1}(f)=\langle\overline{\Lambda}_{\sigma_1}(f),f\rangle$ for the damaged wire is always less than the Dirichlet Energy $\mathbb{E}_{\sigma_0}(f)=\langle\overline{\Lambda}_{\sigma_0}(f),f\rangle$ evaluated in presence of the healthy wire. In Table~\ref{tab:ex2} we report the results of the numerical simulations, which agree with Theorem~\ref{monothm}.
\begin{table}[htp]
\centering
\[
\begin{array}{cccc}
\toprule
f\,\text{(V)} & \mathbb{E}_{\sigma_0} (f)\,\text{(W)}& \mathbb{E}_{\sigma_1}(f)\,\text{(W)} & \mathbb{E}_{\sigma_0}(f)-\mathbb{E}_{\sigma_1}(f)\,\text{(W)} \\
\midrule
100x & 6.6610\times 10^5 & 6.2682\times 10^5 & 3.9284\times 10^4 \\
300x & 3.4927\times 10^6 & 3.3531\times 10^6 & 1.3951\times 10^5 \\
500x & 7.4529\times 10^6 & 7.2170\times 10^5 & 2.3591\times 10^5 \\
0.1\sin(\theta) & 1.4723\times 10^6 & 1.4082\times 10^6 & 6.4086\times 10^4 \\
0.3\sin(\theta) & 7.4530\times 10^6 & 7.2645\times 10^6 & 1.8850\times 10^5 \\
0.5\sin(\theta) & 1.6446\times 10^7 & 1.6139\times 10^7 & 3.0696\times 10^5 \\
0.1\sin(\theta)-0.2\cos(2\theta) & 7.7056\times 10^6 & 7.4816\times 10^6 & 2.2401\times 10^5 \\
100\exp{(x^2+2y)} & 1.9288\times 10^6 & 1.8513\times 10^6 & 7.7496\times 10^4 \\
300\exp{(x^2+2y)} & 9.8394\times 10^6 & 9.6152\times 10^6 & 2.2418\times 10^5 \\
500\exp{(x^2+2y)} & 2.2051\times 10^7 & 2.1685\times 10^7 & 3.6611\times 10^5 \\
\bottomrule
\end{array}
\]
\caption{Numerical evaluation of $\mathbb{E}_{\sigma_0}(f)$ and $\mathbb{E}_{\sigma_1}(f)$ for the case of the damaged petal. The difference $\mathbb{E}_{\sigma_0}(f)-\mathbb{E}_{\sigma_1}(f)$ is greater than $0$.}
\label{tab:ex2}
\end{table}

\section{Conclusions}
\label{conclusions_sec}
The contribution of this paper is focused on the Monotonicity Principle for nonlinear materials. Specifically, it covers the cases of (i) nonlinear materials with piecewise growth exponents and (ii) nonlinear materials with a  growth exponent in $(1,2)$.

These results make it possible to extend a fast imaging method based on the Monotonicity Principle, to the wide class of problems with two or more materials, where at least one is nonlinear.

The treatment is very general and makes it possible to model a wide variety of practical configurations, such as Superconductors (SC), or Perfect Electrical Conductors(PEC), or Perfect Electrical Insulators (PEI).

The key results for proving the Monotonicity Principle in this setting are the convergence of the scalar potential with respect to the boundary data $f$ (Proposition \ref{X_trace_inequality_lem}), and 
the Proposition \ref{gateauxprop}, which gives the connection between the Dirichlet Energy and the measured boundary data.

Numerical results related to a superconducting cable confirm the validity of the Monotonicity Principle.

\section*{Acknowledgements}
This work has been partially supported by the Italian Ministry of University and Research (projects n. 2022SLTHCE and n. 2022Y53F3X, PRIN 2022), and by the National Group for Mathematical Analysis Probability and their Applications (GNAMPA) of the Italian National Institute of Higher Mathematics (INdAM).

\section*{Data availability statement}
No new data were created or analysed in this study.

\bibliographystyle
%{iopart-num}
%{plain}
%{unsrt}
%{alpha}
%{ieeetr}
{siam}
%{abbrv}
\bibliography{biblioCFPPT}

\end{document}